\documentclass[a4paper,12pt,reqno]{amsart}

\pdfoutput=1
\calclayout

\numberwithin{equation}{section}

\usepackage[T1]{fontenc}
\usepackage[english]{babel}
\usepackage{amsthm}
\usepackage{amsfonts}
\usepackage{amsmath}
\usepackage{amssymb}

\usepackage{mathtools}
\usepackage{enumerate}

\usepackage{enumitem}

\usepackage{array}
\usepackage{longtable}
\newcolumntype{C}[1]{>{\centering\arraybackslash}p{#1}}

 \usepackage{hyperref}

\DeclareMathOperator{\rank}{rank}

\DeclareMathOperator{\Disc}{\Delta}

\DeclareMathOperator{\rad}{rad}

\newcommand{\E}{\mathcal{E}}

\newcommand{\Q}{\mathbb{Q}}

\newcommand{\Z}{\mathbb{Z}}
\newcommand{\F}{\mathbb{F}}
\newcommand{\p}{\mathfrak{p}}
\newcommand{\n}{\mathfrak{n}}
\newcommand{\h}{\hat{h}}
\newcommand{\m}{\mathfrak{m}}

\newcommand{\x}{\mathbf{x}}

\newcommand{\B}{\mathbf{B}}

\renewcommand{\epsilon}{\varepsilon}
\renewcommand{\leq}{\leqslant}
\renewcommand{\geq}{\geqslant}

\newtheorem{theorem}{Theorem}[section]

\newtheorem{lemma}[theorem]{Lemma}

\newtheorem{conjecture}[theorem]{Conjecture}

\theoremstyle{definition}

\newtheorem{definition}[theorem]{Definition}

\begin{document}

\author{Stephanie Chan}
\address{Institute of Science and Technology Austria\\
Am Campus 1\\
3400 Klosterneuburg\\
Austria}
\email{stephanie.chan@ist.ac.at}
\title[Szpiro ratio of elliptic curves with prescribed torsion]{Almost all elliptic curves with prescribed torsion have Szpiro ratio close to the expected value}
\date{\today}

\begin{abstract}
We demonstrate that almost all elliptic curves over $\Q$ with prescribed torsion subgroup, when ordered by naive height, have Szpiro ratio arbitrarily close to the expected value. We also provide upper and lower bounds for the Szpiro ratio that hold for almost all elliptic curves in certain one-parameter families. 
The results are achieved by proving that, given any multivariate polynomial within a general class, the absolute value of the polynomial over an expanding box is typically bounded by a fixed power of its radical.
The proof adapts work of Fouvry--Nair--Tenenbaum, which shows that almost all elliptic curves have Szpiro ratio close to $1$. 
\end{abstract}


\maketitle

\setcounter{tocdepth}{1}
\tableofcontents
\section{Introduction}

Every elliptic curve $E$ over $\Q$ admits a unique minimal short Weierstrass equation of the form $y^2=x^3+Ax+B$, where $A$ and $B$ are integers, and $\gcd(A^3, B^2)$ is not divisible by any twelfth powers. We define the naive height of $E$ to be 
\begin{equation}\label{eq:naiveheight}
H(E)\coloneqq\max\left\{4|A|^3,27B^2\right\}.\end{equation}

The Szpiro ratio of $E$ is defined by
\[\sigma(E)\coloneqq\frac{\log |\Disc_{\min}(E)|}{\log N(E)},\]
where $\Disc_{\min}$ is the minimal discriminant of $E$ and $N(E)$ is the conductor of $E$. The prime divisors of $\Disc_{\min}$ and the prime divisors of $N(E)$ coincide, but it is possible that a prime divides $\Disc_{\min}$ with a very large multiplicity. A deep conjecture made by Szpiro asserts that the ratio $\sigma(E)$ should be absolutely bounded.

\begin{conjecture}[Szpiro {\cite[Conjecture~1]{Szpiroconj}}]
For every $\epsilon>0$, there exists some constant $c>0$ depending only on $\epsilon$, such that all elliptic curves $E$ over $\Q$ satisfy
\[|\Disc_{\min}(E)|\leq cN(E)^{6+\epsilon}.\]
\end{conjecture}
Szpiro’s conjecture is closely related to the $abc$ conjecture~\cite{OesABC}, see for example~\cite[Chapter VIII.11]{Silverman}.

We are interested in statistical versions of Szpiro's conjecture in certain families of elliptic curves. 
Fouvry, Nair, and Tenenbaum~\cite[Théorème~1 and Théorème~2]{FNT} showed that for any given $\epsilon>0$, the Szpiro ratio of almost all $\Q$-isomorphism classes of elliptic curves, whether ordered by naive height or by absolute discriminant, is at most $1+\epsilon$. 
 Note that this however does not imply the same is true for thin families, and indeed, as we will see in Theorem~\ref{theorem:szpiroell}, there are natural families with typical Szpiro ratio much larger than $1$.
 Wong \cite[Theorem~3(b)]{WongS} obtained a similar result for Frey curves $y^2=x(x+a)(x+b)$, showing that the Szpiro ratio is at most $2+\epsilon$ outside of an exceptional set which has density $0$ in this family when ordered by naive height.

In this paper, we study certain families of elliptic curves ordered by naive height $H(E)$.
We denote the radical of any integer $n$ by
\[\rad(n)\coloneqq \prod_{p\mid n} p.\]
Since $N(E)$ is bounded below by $\rad(\Disc_{\min}(E))$, to show that $|\Disc_{\min}(E)|\leq N(E)^{\beta}$ for some $\beta>1$, it suffices to show that 
 $|\Disc_{\min}(E)|\leq \rad(\Disc_{\min}(E))^{\beta}$. We will look at families of elliptic curves with discriminants expressible in terms of the values of some bivariate polynomial.
 A key result is Theorem~\ref{theorem:Szpiroirr}, which shows that, for an irreducible polynomial $F\in\Z[\x]$ taking a certain form, the value of $|F(\x)|$ is almost always bounded polynomially in terms of its radical $\rad(F(\x))$ in an expanding box.
 
\begin{theorem}\label{theorem:Szpiroirr}
Suppose $F\in\Z[\x]$ is a polynomial in $k\geq 1$ variables $x_1,\dots,x_k$, is irreducible in $\Q[\x]$, and has total degree $d\geq 1$. Let $B_1,\dots ,B_k\geq 3$.
Assume that there exists some $j$ such that the $x_j^{d}$-coefficient of $F(\x)$ is non-zero and 
\begin{equation}\label{eq:lopbox}
 B_1^{e_1}\dots B_k^{e_k}\ll (\max B_i)^{d-1}B_j
\end{equation}
for all integer $k$-tuple $(e_1,\dots,e_k)$ such that the $x_1^{e_1}\dots x_k^{e_k}$-coefficient of $F$ is non-zero.
 Fix a constant $\beta>\max\left\{1,d-2\right\}$. Then whenever 
 \begin{equation}\label{eq:ABlog}
\log \max B_i\ll\log\min B_i
\end{equation}
 holds, we have
\[
 \frac{1}{B_1\dots B_k}\#\left\{\x\in\Z^k: |x_i|\leq B_i,\ |F(\x)|\geq\rad(F(\x))^{\beta}\right\}
 \ll \frac{(\log\log B_1)^{\max\left\{1,d-1\right\}}}{(\log B_1)^{\delta}},
\]
where
\[\delta=\begin{cases}
\hfil 1&\text{if }d\leq 2,\\
1-\frac{2}{d+1}&\text{if }d\geq 3,
\end{cases}
\]
and the implied constant depends at most on $\beta$, $F$ and the implied constants in~\eqref{eq:lopbox} and~\eqref{eq:ABlog}.
\end{theorem}

Notice that trivially $|F(\x)|\geq\rad(F(\x))$ for any $\x\in\Z^k$, so the lower bound $\max\left\{1,d-2\right\}$ for $\beta$ in Theorem~\ref{theorem:Szpiroirr} is sharp when $d\leq 3$.
When $k=2$, as long as one of the $x_1^d$- and $x_2^d$-coefficients of $F$ is non-zero, there exists a choice of $j$ such that~\eqref{eq:lopbox} holds.

We prove Theorem~\ref{theorem:Szpiroirr} in Section~\ref{section:maintech} by generalising the technique of Fouvry--Nair--Tenenbaum~\cite{FNT}. If we apply Theorem~\ref{theorem:Szpiroirr} to $F(x,y)=4x^3+27y^2$, we recover~\cite[Théorème~2]{FNT} albeit with a more restrictive condition~\eqref{eq:ABlog} as stated. 

In Section~\ref{section:oneparfam}, we use Theorem~\ref{theorem:Szpiroirr} to deduce the following statistical result on the Szpiro ratio of elliptic curves within certain one-parameter families, ordered by naive height.
\begin{theorem}\label{theorem:szpiroell}
Let $f,g\in\Q[t]$ be coprime polynomials that satisfy
 \begin{equation}\label{eq:assumefg}
 \max\left\{\frac{1}{4}\deg f,\ \frac{1}{6}\deg g\right\}\in\left\{1,\ 2,\ 3,\ \dots \right\}\cup \left\{\frac{1}{2},\ \frac{1}{3},\ \frac{1}{4},\ \dots \right\}.\end{equation}
Let $\mathcal{E}=(\mathcal{E}_t)_{t\in\Q}$ be the family of elliptic curves defined by
\begin{equation}\label{eq:paraE}
\mathcal{E}_t:y^2=x^3+f(t)x+g(t),\ t\in\Q.\end{equation}
Fix constants $\beta_1$ and $\beta_2$ such that $\beta_1<\lambda_{1}(f,g)\leq \kappa_{1}(f,g)<\beta_2$, where $\lambda_{1}$ and $\kappa_{1}$ are defined in~\eqref{eq:lambda} and~\eqref{eq:kappa}.
Then for all $H\geq 3$ we have
\[ \frac{\#\left\{E/\Q: E\in\mathcal{E},\ H(E)\leq H,\ \sigma(E)\notin (\beta_1,\beta_2)\right\}}{\#\left\{E/\Q: E\in\mathcal{E},\ H(E)\leq H\right\}}\ll \frac{(\log\log H)^2}{(\log H)^{\frac{1}{2}}},\]
where the implied constant depends at most on $\beta_1$, $\beta_2$, $f$ and $g$.
\end{theorem}

The constants $\lambda_{1}(f,g)$ and $\kappa_{1}(f,g)$ come from studying the factorisation of the discriminant of $\mathcal{E}_t$, viewed as a polynomial in $t$. As an example, suppose that $f,g\in \Z[t]$ are coprime in $\Q[t]$, and $n=\max\left\{\frac{1}{4}\deg f,\ \frac{1}{6}\deg g\right\}$ is a positive integer. Write $t=a/b$, where $a$ and $b$ are coprime integers. Then $\mathcal{E}_t$ is $\Q$-isomorphic to
\[y^2=x^3+F(a,b)x+G(a,b),\]
where $F(a,b)\coloneqq b^{4n}f(a/b)$ and $G(a,b)\coloneqq b^{6n}g(a/b)$ are integral binary forms that are coprime because of our choice of $n$. Although this model is not necessarily minimal, it turns out that it is not very far from being so. The discriminant with respect to this model is a constant multiple of $D(a,b)\coloneqq b^{12n}d(a/b)$, which has degree $12n$. 
Suppose that $D_i$ are the distinct irreducible factors of $D$ in $\Q[x,y]$. We expect $D_i(a,b)$ to be close to being squarefree. To obtain the value of $\lambda_{1}(f,g)$, we estimate $|\Disc_{\min}(E)|$ by $|D(a,b)|$ and $N(E)$ by $\prod D_i(a,b)$. As $a$ and $b$ vary in an expanding box of equal side lengths, $\frac{\log |D_i(a,b)|}{\log |D(a,b)|}$ is usually close to $\frac{\deg D_i}{\deg D}$. This suggests that $\lambda_{1}(f,g)=\frac{\deg D}{\sum_i \deg D_i}$ is the expected value of the Szpiro ratio in this family.
When the degree of $D_i$ is greater than $3$, Theorem~\ref{theorem:Szpiroirr} only provides an upper bound $|D_i(a,b)|\leq (\rad(D_i(a,b))^{\deg D_i-2+\epsilon}$ for almost all $(a,b)$, so $\kappa_{1}(f,g)$ is defined correspondingly with the extra weight $\max\{1,\deg D_i-2\}$ in mind.

Harron--Snowden~\cite{HarronSnowden} (see also Cullinan--Kenney--Voight~\cite{CKV}) found the order of magnitude for the number of elliptic curves over $\Q$ of bounded height with any given torsion subgroup $G$. In their work, they utilised the fact that for a majority of permissible groups $G$, there exists $f,g\in\Q[t]$ such that any elliptic curve $E$ with torsion subgroup $E(\Q)_{\mathrm{tors}}\cong G$ is $\Q$-isomophic to some $\mathcal{E}_t$ of the form~\eqref{eq:paraE}. This provides us with the setup for dealing with families of elliptic curves with prescribed torsion.

As a consequence of Theorem~\ref{theorem:szpiroell}, we show that almost all elliptic curves with any prescribed torsion subgroup $G$ have Szpiro ratio arbitrarily close to the expected value $\beta_G$, when ordered by naive height.
\begin{theorem}\label{theorem:SzpiroTorsheight}
Let $\epsilon>0$.
Let $G$ be a possible torsion subgroup for elliptic curves over $\Q$.
Define\[\beta_G\coloneqq\begin{cases}
\hfil 1&\text{if }G= C_1,\\
\hfil\frac{3}{2}&\text{if }G= C_2,\\
\hfil 2&\text{if }G= C_3,\text{ or }C_2\times C_2,\\
\hfil\frac{12}{5}&\text{if }G= C_4,\\
\hfil 3&\text{if }G= C_5,\ C_6,\text{ or } C_2\times C_4,\\
\hfil 4&\text{if }G= C_7,\ C_8,\text{ or } C_2\times C_6,\\
\hfil\frac{9}{2}&\text{if }G= C_9,\text{ or }C_{10},\\
\hfil\frac{24}{5}&\text{if }G= C_{12},\text{ or }C_2\times C_8,
\end{cases}\]
where $C_n$ denotes the cyclic group of order $n$.
Then for all $H\geq 3$ we have
\[
 \frac{\#\left\{E/\Q: E(\Q)_{\mathrm{tors}}\cong G,\ H(E)\leq H,\ \sigma(E)\notin (\beta_G-\epsilon,\beta_G+\epsilon)\right\}}{\#\left\{E/\Q: E(\Q)_{\mathrm{tors}}\cong G,\ H(E)\leq H\right\}}\ll \frac{(\log\log H)^{2}}{(\log H)^{\frac{1}{2}}},
\]
where the implied constant depends at most on $\epsilon$.
\end{theorem}

Since almost all elliptic curves have trivial torsion over $\Q$, Theorem~\ref{theorem:SzpiroTorsheight} in the case where $G$ is trivial, can be seen as a reformulation of \cite[Théorème~2]{FNT} of Fouvry--Nair--Tenenbaum. One can also recover Wong's result on the Szpiro ratio of Frey curves \cite[Theorem~3(b)]{WongS} from the case where $G=C_2\times C_2$.
For any possible $G$, every irreducible factor of the discriminant polynomial arising from the parameterisation has degree at most $3$. Consequently, we can establish, up to $\epsilon$, matching upper and lower bounds for the Szpiro ratio for almost all curves with given torsion subgroup $G$. In the cases where we apply Theorem~\ref{theorem:szpiroell}, the value of $\beta_G$ is $\lambda_1(f,g)=\kappa_1(f,g)$.
As a remark, we observe that our values of $\beta_G$ in Theorem~\ref{theorem:SzpiroTorsheight}
 agree with the lower bounds obtained in~\cite{BarriosSzp} for the modified Szpiro ratio $\frac{\log H(E)+O(1)}{\log N(E)}$ .

Lang \cite[p.\,92]{LangDioph} conjectured that there exists an absolute constant $c$ such that for any elliptic curve $E/\Q$ and for any non-torsion point $P\in E(\Q)$, the canonical height of $P$ is bounded below by
\begin{equation}\label{eq:Lang}
\h(P)\geq c\log |\Disc_{\min}(E)|.
\end{equation}
Hindry--Silverman~\cite{HS} showed that Szpiro’s conjecture implies Lang's conjecture. Quantitatively, their result implies that $c$ in~\eqref{eq:Lang} can be chosen to only depend on the Szpiro ratio. 
Combining our results with~\cite[Theorem~0.3]{HS}, we see that Lang's conjecture holds for almost all elliptic curves in the families covered by Theorem~\ref{theorem:szpiroell} and Theorem~\ref{theorem:SzpiroTorsheight}.

\begin{theorem}\label{theorem:langstat}
\begin{enumerate}
\item Let $f,g\in\Q[t]$ be coprime polynomials that satisfy~\eqref{eq:assumefg}.
Let $\mathcal{E}=(\mathcal{E}_t)_{t\in\Q}$ be as defined in~\eqref{eq:paraE}.
There exists a constant $c$ depending only on $\kappa_{1}(f,g)$ defined in~\eqref{eq:kappa}, such that for all $H\geq 3$, we have
\[ \frac{\#\left\{E/\Q: \begin{array}{l}
E\in\mathcal{E},\ H(E)\leq H,\\
\eqref{eq:Lang}\text{ fails for some }P\in E(\Q)\setminus E(\Q)_{\mathrm{tors}}
 \end{array}\right\}}{\#\left\{E/\Q: E\in\mathcal{E},\ H(E)\leq H\right\}}\ll \frac{(\log\log H)^2}{(\log H)^{\frac{1}{2}}},\]
where the implied constant depends at most on $f$ and $g$.
\item Let $G$ be a possible torsion subgroup for elliptic curves over $\Q$. \label{tors}
There exists an absolute constant $c$ such that for all $H\geq 3$, we have
\[
 \frac{\#\left\{E/\Q: \begin{array}{l}E(\Q)_{\mathrm{tors}}\cong G,\ H(E)\leq H,\\
 \eqref{eq:Lang}\text{ fails for some }P\in E(\Q)\setminus E(\Q)_{\mathrm{tors}}\end{array}\right\}}{\#\left\{E/\Q: E(\Q)_{\mathrm{tors}}\cong G,\ H(E)\leq H\right\}}\ll \frac{(\log\log H)^{2}}{(\log H)^{\frac{1}{2}}},
\]
and the implied constant is absolute.
\end{enumerate}
\end{theorem}
Le Boudec studied statistical versions of Lang's conjecture in \cite{LBql,LBl}. In the family of all elliptic curves defined over $\Q$ with positive rank, he showed that on fixing $c<7/24$, \eqref{eq:Lang} holds in a density $1$ subfamily \cite[Theorem 1]{LBl}.  This result is stronger than Theorem~\ref{theorem:langstat}\eqref{tors} in the case where $G$ is trivial, as $c$ can be taken much larger than what follows from applying the bound in \cite[Theorem~0.3]{HS}. 

Another immediate consequence of Theorem~\ref{theorem:szpiroell}, Theorem~\ref{theorem:SzpiroTorsheight} and~\cite[Theorem~0.7]{HS} is the following.
\begin{theorem}\label{theorem:intptbd}
Let $S$ be a finite set of places of $\Q$ including the archimedean place and $R_S$ be the ring of $S$-integers of $\Q$.
Let $E(R_S)$ denote the set of solutions $x,y\in R_S$ to the minimal short Weierstrass equation for $E/\Q$.
\begin{enumerate} 
\item Let $f,g\in\Q[t]$ be coprime polynomials that satisfy~\eqref{eq:assumefg}.
Let $\mathcal{E}=(\mathcal{E}_t)_{t\in\Q}$ be as defined in~\eqref{eq:paraE}.
There exists a constant $c$ depending only on $\kappa_{1}(f,g)$ defined in~\eqref{eq:kappa}, such that for all $H\geq 3$, we have
\[ \frac{\#\left\{E/\Q: E\in\mathcal{E},\ H(E)\leq H,\ \# E(R_S) \geq c^{|S|+\rank E(\Q)}\right\}}{\#\left\{E/\Q: E\in\mathcal{E},\ H(E)\leq H\right\}}\ll \frac{(\log\log H)^2}{(\log H)^{\frac{1}{2}}},\]
where the implied constant depends at most on $f$ and $g$.
\item Let $G$ be a possible torsion subgroup for elliptic curves over $\Q$.
There exists an absolute constant $c$ such that for all $H\geq 3$, we have
\[
 \frac{\#\left\{E/\Q: E(\Q)_{\mathrm{tors}}\cong G,\ H(E)\leq H,\ \# E(R_S) \geq c^{|S|+\rank E(\Q)}\right\}}{\#\left\{E/\Q: E(\Q)_{\mathrm{tors}}\cong G,\ H(E)\leq H\right\}}\ll \frac{(\log\log H)^{2}}{(\log H)^{\frac{1}{2}}},
\]
and the implied constant is absolute.
\end{enumerate}
\end{theorem}
It is probable that, with additional effort, stronger versions of Theorem~\ref{theorem:intptbd} could be derived from Alp\"{o}ge and Ho's upper bound for $\# E(R_S)$ in \cite[Theorem~1.2]{AlpogeHo}, but this will not be explored here.

\subsection*{Acknowledgement}
The author would like to thank Tim Browning and Matteo Verzobio for useful comments and discussions.

\section{Proof of Theorem~\texorpdfstring{\ref{theorem:Szpiroirr}}{1.2}}\label{section:maintech}
The goal of this section is to prove Theorem~\ref{theorem:main1}, which implies Theorem~\ref{theorem:Szpiroirr}. 

Fix a positive integer $k$. Given a $k$-tuple of positive real numbers $\B=(B_1,\dots,B_k)$, consider $\x=(x_1,\dots,x_k)\in\Z^k$ such that $|x_i|\leq B_i$ for all $1\leq i\leq k$.
Let 
\begin{align*}
X&\coloneqq\max\left\{ B_1^{e_1}\dots B_k^{e_k}: \text{the } x_1^{e_1}\dots x_k^{e_k}\text{-coefficient of $F$ is non-zero}\right\},\\
Y&\coloneqq\max B_i,\\
Z&\coloneqq\min B_i.
\end{align*}
Fix a polynomial $F\in\Z[\x]$ which satisfies the following assumptions:
\begin{enumerate}[label=(A{\arabic*})]
\item $F$ has total degree $d\geq 1$ and has no repeated polynomial factors over $\Q$. \label{a:doubleroot} 
\item The $x_1^d$-coefficient of $F(\x)$ is non-zero. \label{a:specialpoly}
\item There exists $i$ such that $B_i=Y$ such that $F$ and $\frac{\partial }{\partial x_i}F$ have no common factor in the ring $\Z[\x]$. \label{a:codim2} 
\end{enumerate}

\begin{theorem}\label{theorem:main1}
Assume that $F\in\Z[\x]$ satisfies~\ref{a:doubleroot},~\ref{a:specialpoly} and~\ref{a:codim2}.
Let $B_1,\dots,B_k\geq 3$ such that
\begin{equation}\label{eq:loggrowth}
\log Y\ll \log Z
\end{equation}
and
\begin{equation}\label{eq:loglb}
X\ll B_1Y^{d-1}.
\end{equation}
If $d\geq 3$, further assume that $\tau> \frac{d-1}
{d-r}$, 
where $\tau\coloneqq \tau(d-1)$ denotes the number of divisors of $d-1$, and $r$ is the number of distinct irreducible factors of $F$ over $\Q$.
Fix a constant $\beta>\max\left\{1,d-2\right\}$.
Then
\[
 \frac{1}{B_1\dots B_k}\#\left\{\x\in\Z^k: |x_i|\leq B_i,\ |F(\x)|\geq \rad(F(\x))^{\beta}\right\}\ll \frac{(\log\log Z)^{\max\left\{1,d-r\right\}}}{(\log Z)^{\delta_d(r)}},
\]
where
\begin{equation}\label{eq:defdelta}
\delta_d(r)\coloneqq
\begin{cases}
\hfil 1&\text{if }d=1\text{ or }2,\\ \displaystyle
\frac{(d-1)(\frac{d-r}{d-1}-\frac{1}{\tau})}{1+(d-1)(1-\frac{1}{\tau})}&\text{if }d\geq 3,
\end{cases}
\end{equation}
and the implied constant depends at most on $\beta$, $F$, and the implied constants in~\eqref{eq:loggrowth} and~\eqref{eq:loglb}.
\end{theorem}

 Note that in particular $\tau\geq 2$ in~\eqref{eq:loggrowth} whenever $d\geq 3$.
Theorem~\ref{theorem:Szpiroirr} follows from Theorem~\ref{theorem:main1} by putting in $r=1$ and relabelling the variables such that~\ref{a:specialpoly} and~\eqref{eq:loglb} are satisfied. The assumption that $F$ is irreducible in $\Q[\x]$ and $d\geq 1$ implies~\ref{a:doubleroot} and~\ref{a:codim2}.

\subsection{Preliminaries}
We collect some results which are the ingredients of the proof of Theorem~\ref{theorem:main1}.
 
\begin{lemma}[{Geometric sieve~\cite[Lemma~2.1]{BHB}}]\label{lemma:geoms}
Let $B_1,\dots ,B_k,H, M\geq 2$ and let $f_1,\dots, f_r\in \Z[\x]$ be polynomials with no common factor in the ring $\Z[\x]$, and having degrees at most $d$ and coefficients with absolute value less than $H$. Let
$V\coloneqq B_1\dots B_k$.
Then
\[\#\left\{\x\in \Z^k:\begin{array}{l}
|x_i|\leq B_i\text{ for }i\leq k,\\ \exists p>M,\ p\mid f_j(\x)\text{ for }j\leq r \end{array}\right\}\ll \frac{V\log (VH)}{M\log M}+\frac{V\log (VH)}{\min B_i},\]
where the implied constant is only allowed to depend on $d$ and $k$.
\end{lemma}

\begin{lemma}[{Larger sieve~\cite{Hooleysieve}}]\label{lemma:largersieve}
Suppose that $\mathcal{S}$ is a set of prime powers. Let $\Omega_q\subset \Z/q\Z$ for all $q\in\mathcal{S}$.
Then
\begin{equation}\label{eq:largersieve}
\#\left\{x\in \Z: |x|\leq X,\ x\bmod q\in\Omega_q\text{ for all }q\in \mathcal{S}\right\}\leq \frac{\sum_{q\in\mathcal{S}}\Lambda(q)-\log X}{\sum_{q\in\mathcal{S}}\frac{\Lambda(q)}{\#\Omega_q}-\log X},\end{equation}
provided that the denominator is positive.
\end{lemma}

\begin{lemma}[{Large sieve~\cite[Theorem~4.1]{Kowalskisieve}}]\label{lemma:largesieve}
Let $\Omega_p\subset \F_p^k$ for all prime $p$. Then
\[\#\left\{\x\in \Z^k: |x_i|\leq B_i,\ \x\bmod p\in\Omega_p\text{ for all }p\leq Q\right\}\ll \frac{\prod_{i=1}^k(B_i+Q^2)}{L(Q)},\]
where
\[L(Q)=\sum_{1\leq n\leq Q}\mu^2(n)\prod_{p\mid n} \frac{p^k-\#\Omega_p}{\#\Omega_p}.\]
\end{lemma}

\begin{lemma}[Tenenbaum {\cite[Lemme~2]{Tenenbaumsmooth}}]\label{lemma:T}
Suppose $\lambda_1>0$ and $0<\lambda_2<2$. Let $f$ be a multiplicative function such that
\[0\leq f(p^v)\leq \lambda_1\lambda_2^{v}\text{ for every }p\geq 2,v\geq 1.\]
Then uniformly for all $X,t, T\geq 2$, we have
\[\sum_{\substack{1\leq n\leq X\\ n_t\geq T}}f(n)\ll \frac{X}{\log X}\exp\left(-\lambda_3\frac{\log T}{\log t}+\sum_{p\leq X}\frac{f(p)}{p}\right),\]
where $\lambda_3$ depends only on $\lambda_2$, and the implied constant depends only on $\lambda_1$ and $\lambda_2$.
\end{lemma}

\subsection{Level of distribution estimates}
Define 
\[\rho(q)\coloneqq\frac{1}{q^k}\#\left\{\x\in(\Z/q\Z)^k : F(\x)\equiv 0\bmod q\right\}.\]
Also define
\[
T(q,\B)\coloneqq\#\left\{\x\in\Z^k:|x_i|\leq B_i,\ q\mid F(\x)\right\}.
\]
We have an upper bound 
\[T(q,\B)\ll\rho(q) q^k\prod_{i=1}^{k} \left(\frac{B_i}{q}+1\right) .\]
When $q\leq Z$, clearly
\begin{equation}\label{eq:easylevel}
T(q,\B)\ll\rho(q) B_1\dots B_k.
\end{equation}

\begin{lemma}
Assume that $F\in\Z[\x]$ satisfies~\ref{a:doubleroot}. Then
\begin{equation}\label{eq:rhobounds}
 \rho(p^v)\ll \begin{cases}
 \hfil p^{-1}&\text{if }v=1,\\
 \hfil p^{-2}&\text{if }2\leq v\leq 2d,\\
 \hfil p^{-\frac{v}{d}}&\text{if }v>2d,
 \end{cases}
\end{equation}
where the implied constant depends at most on $F$.
\end{lemma}
\begin{proof}
Lang--Weil estimates~\cite{LangWeil} implies
\[\rho(p)\ll p^{-1}.\]
Under the assumption~\ref{a:doubleroot}, Lemma~2.6 and Lemma~2.8 in~\cite{CKPS} provide the upper bound
\[\rho(p^v)\ll \min\left\{p^{-2},p^{-\frac{v}{d}}\right\},\]
when $v\geq 2$.
\end{proof}

\begin{lemma}[{\cite[Lemma~2.5]{CKPS}}]\label{lemma:cheb}
Assume that $F\in\Z[\x]$ satisfies~\ref{a:doubleroot}.  Then 
\[\sum_{p\leq X}\rho(p)=r\log\log X+O(1),\]
where $r$ is the number of distinct irreducible factors of $F$ in $\Q[\x]$.
The implied constant depends only on $F$.
\end{lemma}

\subsection{Splitting into cases}
To prove Theorem~\ref{theorem:main1}, we will largely follow the proof of Théorème~2 in work of Fouvry--Nair--Tenenbaum~\cite[Section~3]{FNT} with some necessary generalisations.

Define
\[
r(n,\B)\coloneqq\#\left\{\x\in\Z^k: |x_i|\leq B_i,\ F(\x)=n\right\}.\]
Write
\[V\coloneqq B_1\dots B_k.\]
Take the following parameters
\begin{align}\label{eq:parameters}
t&\coloneqq (\log B_1)^2,&
T&\coloneqq \left(\frac{B_1}{\log B_1}\right)^{1-\frac{ d-2}{\beta}},&
T_1&\coloneqq Y(\log Y)^{1-\delta},
\end{align}
where $\delta\coloneqq \delta_d(r)$ as in~\eqref{eq:defdelta}.
We want to bound
\begin{equation}\label{eq:mainsum}
\#\left\{\x\in\Z^k: |x_i|\leq B_i,\ |F(\x)|\geq \rad(F(\x))^{\beta}\right\}
=\sum_{\substack{n\in\Z \\ |n|\geq \rad(n)^{\beta}}} r(n,\B).\end{equation}
Given $n\in\Z$, we write
\[
n_t\coloneqq \prod_{p\leq t} p^{v_p(n)}.
\]
We will split the sum~\eqref{eq:mainsum} into the following cases:
\begin{enumerate}[label=(C{\arabic*})]
\item $|n|\leq \frac{B_1}{\log B_1}$. \label{C1}
\item $n_t\geq T$. \label{C2}
\item $p^2\mid n$ for some $t<p\leq T_1$. \label{C:medprimes}
\item $d\geq 3$ and $p^{d-1}\mid n$ for some $p>T_1$.\label{C:largeprimes}
\end{enumerate}

We now deduce that~\ref{C:largeprimes} is enough to cover the case when~\ref{C1},~\ref{C2},~\ref{C:medprimes} all fail.
Assume that $|n|> \frac{B_1}{\log B_1}$, $n_t<T$, $p^2\nmid n$ for every $t<p\leq T_1$. If $d\leq 2$, $p^2\mid n\ll Y^d$ implies that $p\ll Y$, so this contradicts with $p>T_1$ when $Y$ is large enough.
Suppose $d\geq 3$ and recall that $|n|\geq \rad(n)^{\beta}$.
If $v_p(n)\leq d-2$ for all $p>t$, then 
\[\frac{|n|}{T}<\frac{|n|}{n_t}\leq \prod_{p\geq t} p^{ d-2}\leq\rad(n)^{d-2}\leq |n|^{\frac{ d-2}{\beta}},\]
which contradicts with $|n|> \frac{B_1}{\log B_1}$ and $T=(\frac{B_1}{\log B_1})^{1-\frac{ d-2}{\beta}}$.
Therefore we end up with $v_p(n)\geq d-1$ for some prime $p>T_1$, which corresponds to~\ref{C:largeprimes}.

\subsection{Case~\texorpdfstring{\ref{C1}}{C1}}
\begin{lemma} \label{lemma:C1}
Assume that $F\in\Z[\x]$ satisfies~\ref{a:specialpoly}. Let $F\in\Z[\x]$ be a non-constant polynomial.
For $B_1,\dots,B_k\geq 3$, we have
\[\sum_{|n|\leq B_1/\log B_1}r(n,\B)\ll\frac{V}{\log B_1},\]
where the implied constant depends at most on $F$.
\end{lemma}
\begin{proof}
Fixing $x_2,\dots, x_k$ and $n$, we obtain a degree $d$ polynomial in $x_1$, so there are at most $d$ possible $x_1$.
The number of choices of $x_2,\dots, x_k$ is $O(V/B_1)$ and the number of choices of $n$ is $O(B_1/\log B_1)$.
\end{proof}

\subsection{Case~\texorpdfstring{\ref{C2}}{C2}}
\begin{lemma} \label{lemma:smoothpart}
Assume that $F\in\Z[\x]$ satisfies~\ref{a:doubleroot}. 
For $B_1,\dots,B_k,T\geq 3$ and $3\leq t\leq Z$, we have
\begin{equation}\label{eq:C2}
\sum_{\substack{n\\n_t\geq T}}r(n,\B)\ll \frac{V}{\log T}\left(\log t+\frac{t^{1+\frac{1}{d}}}{Z^{\frac{1}{d}}}\frac{(\log X)^2}{\log t}\right),\end{equation}
where the implied constant depends at most on $F$.
\end{lemma}
\begin{proof}
We have
\[\sum_{\substack{n\in\Z \\n_t\geq T}}r(n,\B)\leq \sum_{n}r(n,\B)\frac{\log n_t}{\log T},\]
so we want to estimate
\begin{equation}\label{eq:sumrlognt}
 \sum_{n\in\Z }r(n,\B)\log n_t
 =\sum_{p\leq t}\sum_{\substack{v\geq 1\\p^v\leq X}}T(p^v,\B)\log p^v
=\sum_{p\leq t}\log p\sum_{\substack{v\geq 1\\p^v\leq X}}T(p^v,\B)v.
\end{equation}
Split the inner sum according to whether $p^v\leq Z$ or $p^v>Z$.
If $p^v\leq Z$, we apply~\eqref{eq:easylevel},
\[T(p^v,\B)\ll\rho(p^v) V.\]
By~\eqref{eq:rhobounds}, we have
\[
\sum_{\substack{v\geq 1\\p^v\leq Z}}\rho(p^v)v\ll \frac{1}{p}+\sum_{2\leq v\leq 2d}\frac{v}{p^2}+\sum_{v>2d}\frac{v}{p^{\frac{v}{d}}}
\ll \frac{1}{p}+\frac{d^2}{p^2}+ \frac{d^2}{p^2(\log p)^2}
\ll \frac{1}{p}
,
\]
so
\begin{equation}\label{eq:subsumT1}
\sum_{\substack{v\geq 1\\p^v\leq Z}}T(p^v,\B)v\ll
\frac{1}{p}V.
 \end{equation}

If $p^v>Z$, let $m_p$ denote the integer such that $p^{m_p}\leq Z<p^{m_p+1}$. Since $t\leq Z$, $m_p$ must be positive.
We have by~\eqref{eq:easylevel},
\[T(p^v,\B)\leq T(p^{m_p},\B)\ll\rho(p^{m_p}) V,\]
and by~\eqref{eq:rhobounds},
\[ \rho(p^{m_p})\ll p^{-\frac{m_p}{d}}
\leq 
p^{-\frac{1}{d}(\frac{\log Z}{\log p}-1)}
\leq 
(p/Z)^{\frac{1}{d}}.
\]
Therefore
\begin{equation}\label{eq:subsumT2}\sum_{\substack{v\geq 1\\Z<p^v\leq X}}T(p^v,\B)v\ll
V\sum_{\substack{v\geq 1\\Z<p^v\leq X}}\frac{p^{\frac{1}{d}}v}{Z^{\frac{1}{d}}}
\ll
V\frac{p^{\frac{1}{d}}}{Z^{\frac{1}{d}}}\left(\frac{\log X}{\log p}\right)^2.
 \end{equation}

Combining the two cases~\eqref{eq:subsumT1} and~\eqref{eq:subsumT2}, we can bound~\eqref{eq:sumrlognt} by
\[\ll V\sum_{p\leq t}\left(\frac{\log p}{p}+\frac{p^{\frac{1}{d}}(\log X)^2}{Z^{\frac{1}{d}}\log p}\right)
\ll V\left(\log t+\frac{t^{1+\frac{1}{d}}}{Z^{\frac{1}{d}}}\frac{(\log X)^2}{\log t}\right).\]
Therefore we have shown~\eqref{eq:C2} as required.
\end{proof}

\subsection{Case~\texorpdfstring{\ref{C:medprimes}}{C3}}
\begin{lemma}\label{lemma:C3}
Assume that $F$ and $\frac{\partial }{\partial x_1}F$ have no common factor in the ring $\Z[\x]$. Then for all $B_1,\dots,B_k\geq 3$, we have
\begin{equation}\label{eq:goodlevel}
 \sum_{\substack{n\in\Z \\p^2\mid n\text{ for some }t<p\leq T_1}}r(n,\B)\ll\frac{V\log V}{t\log t}+\frac{V\log V}{Z}+\frac{T_1}{\log T_1}\prod_{i=2}^k B_i,
\end{equation}
where the implied constant depends at most on $F$.
\end{lemma}
\begin{proof}
We split the sum according to whether $\frac{\partial }{\partial x_1}F\bmod p$ is non-zero
\[ \sum_{\substack{n\in\Z \\p^2\mid n\text{ for some }t<p\leq T_1}}r(n,\B)\ll \sum_{t<p\leq T_1}S_1(p)+S_2,\]
where
\begin{align*}
S_1(p)&\coloneqq\#\left\{\x\in \Z^k: |x_i|\leq B_i,\ F(\x)\equiv 0\bmod p^2,\ \frac{\partial }{\partial x_1}F(\x)\not\equiv 0\bmod p\right\},\\
S_2&\coloneqq\#\left\{\x\in \Z^k: |x_i|\leq B_i,\ F(\x)\equiv \frac{\partial }{\partial x_1}F(\x)\equiv 0\bmod p\text{ for some }p>t\right\}.
\end{align*}

We first bound $S_1(p)$. Viewing $x_2,\dots,x_k$ as fixed, we can Hensel lift each solution $t\bmod p$ such that $F(t,x_2,\dots,x_k)\equiv 0\bmod p,\ \frac{\partial }{\partial t}F(t,x_2,\dots,x_k)\not\equiv 0\bmod p$ to a unique solution $x_1\bmod p^2$ such that $F(\x)\equiv 0\bmod p^2$. Since $\frac{\partial }{\partial t}F(t,x_2,\dots,x_k)\not\equiv 0\bmod p$, the degree of $F(t,x_2,\dots,x_k)\bmod p$ in $t$ must be positive, so we can simply bound the number of such $t\bmod p$ by the degree of $F$. We conclude that
\[
S_1(p)\ll \left(\frac{B_1}{p^2}+1\right) \prod_{i=2}^k B_i.
\]
Summing over $t<p\leq T_1$, we have
\begin{equation}\label{eq:bdS1}\sum_{t<p\leq T_1}S_1(p)\ll \sum_{t<p\leq T_1}\left(\frac{B_1}{p^2}+1\right)\prod_{i=2}^k B_i\ll 
\left(\frac{B_1}{t\log t}+\frac{T_1}{\log T_1}\right)\prod_{i=2}^k B_i.\end{equation}

For $S_2$, apply Lemma~\ref{lemma:geoms} to get
\begin{equation}\label{eq:bdS2}
S_2\ll \frac{V\log V}{t\log t}+\frac{V\log V}{Z}.
\end{equation}
Summing the upper bounds from~\eqref{eq:bdS1} and~\eqref{eq:bdS2} gives~\eqref{eq:goodlevel}.
\end{proof}

\subsection{Case~\texorpdfstring{\ref{C:largeprimes}}{C4}}
We will prove Lemma~\ref{lemma:C4} in this subsection, which handles the remaining case~\ref{C:largeprimes}.
\begin{lemma}\label{lemma:C4}
Suppose $d\geq 3$.
Assume that $F\in\Z[\x]$ satisfies~\ref{a:doubleroot} and~\ref{a:specialpoly}. Let $r$ be the number of irreducible factors of $F$ over $\Q$.
For all $B_1,\dots,B_k\geq 3$ satisfying~\eqref{eq:loggrowth} and~\eqref{eq:loglb}, we have
\[\sum_{\substack{n\in\Z\setminus\left\{0\right\} \\ p^{d-1}\mid n\text{ for some }p>T_1}}r(n,\B)
\ll V\frac{\left(\log\log Z\right)^{d-r} }{(\log Z)^{\delta} },\]
where $T_1=Y(\log Y)^{1-\delta}$ and $\delta= \delta_d(r)$ as defined in~\eqref{eq:defdelta}. The implied constant depends at most on $F$.
\end{lemma}

For $\mathbf{b}\coloneqq (b_2,\dots, b_k)\in\Z^{k-1}$,
define $F_{\mathbf{b}}\in\Z[x]$ by
\[F_\mathbf{b}(x)\coloneqq F(x,\mathbf{b}) =F(x,b_2,\dots, b_k).\]
Given a non-zero integer $n$, we may factorise $n$ as $mz^{d-1}$, where $m$ is $(d-1)$-th power free.
If $p^{d-1}\mid n$ for some $p>T_1$, then $p\mid z$, so $|m|\leq \frac{|n|}{T_1^{d-1}}$.
Therefore we have the upper bound
\begin{multline}\label{eq:C4bound1}
\sum_{\substack{n\in\Z\setminus\left\{0\right\} \\ p^{d-1}\mid n\text{ for some }p>T_1}}r(n,\B)\\
\leq \sum_{0<|m|\leq \frac{X}{T_1^{d-1}}}\mu^2(|m|)\sum_{\substack{\mathbf{b}\in\Z^{k-1}\\ |b_i|\leq B_i}}\#\left\{x\in\Z:
\begin{array}{l}
|x|\leq B_1,\\ F_{\mathbf{b}}(x)=mz^{d-1}\text{ for some }z\in\Z
\end{array}\right\}.
\end{multline}

\begin{lemma}\label{lemma:largerapp}
Suppose $F(x)\in\Z[x]$ has degree $d\geq 3$ and let $\ell$ be a positive integer coprime to $d$. Let $m$ be a positive integer, and \[\varrho(m)\coloneqq \#\left\{x\in \Z/m\Z: F(x)\equiv 0\bmod m\right\}.\]
Then for all $B\geq m$, we have
\begin{multline*}
\#\left\{x\in\Z :|x|\leq B,\ F(x)= mz^{\ell}\text{ for some }z\in\Z\right\}\\
\ll
\varrho(m)\left(\frac{B}{m}\right)^{\frac{1}{\tau(\ell)}}\exp\left(\frac{1}{\tau(\ell)}\sum_{\substack{p\mid m}}\frac{\log p}{p}\gcd(\ell,p-1)\right),
\end{multline*}
where the implied constant depends only on leading coefficient and degree of $F$.
\end{lemma}
\begin{proof}
Let $c_0$ be the leading coefficient of $F(x)$.
Let 
\[N(p)\coloneqq \#\left\{(x,z)\in\F_p^2: F(x)\equiv mz^{\ell}\bmod p\right\}.\]
Since $d$ and $\ell$ are coprime, $F(x)-mz^{\ell}$ is irreducible in $\F_p[x,z]$ as long as $p\nmid mc_0$ by~\cite[Chapter~III Theorem~1B]{SchmidtFF}.
Therefore Lang--Weil estimates~\cite{LangWeil} imply that $N(p)=p+O(\sqrt{p})$ when $p\nmid mc_0$.
Let $\alpha_1,\dots,\alpha_{\varrho(m)}$ denote the solutions of $F(x)\equiv 0\bmod m$. Split the set according to $x\bmod m$ to obtain
\begin{multline*}
\#\left\{x\in\Z :
|x|\leq B,\ F(x)=mz^{\ell}\text{ for some }z\in\Z\right\}\\
=\sum_{i=1}^{\varrho(m)}\#\left\{x\in\Z :
\begin{array}{l}
|x|\leq B,\ x\equiv \alpha_i\bmod m,\\ F(x)= mz^{\ell}\text{ for some }z\in\Z\end{array}\right\}.
\end{multline*}
Take $P$ to be a parameter which will be specified later. 
For each fixed $i$, we apply Lemma~\ref{lemma:largersieve} with
\[\Omega_q\coloneqq\begin{cases}
\hfil \{\alpha_i\bmod q\}&\text{if } q\mid m,\\
\hfil \{x\bmod p: F(x)\equiv mz^{\ell}\bmod p\text{ for some }z\}&\text{if }q=p\leq P\text{ and }p\nmid m,
\end{cases}\]
and $\mathcal{S}$ the union of primes $c_0<p\leq P$ and all prime powers dividing $m$.
Note that when $p\nmid m$, given $a\in \F_p^{\times}$, $mz^{\ell}\equiv a\bmod p$ has $\gcd(\ell,p-1)$-many solutions for $z\bmod p$. 
We find that
\[\#\Omega_q=\begin{cases}
\hfil 1&\text{if } q\mid m,\\
\hfil \frac{1}{\gcd(\ell,p-1)}\left(N(p)-\varrho(p)\right)+\varrho(p)&\text{if }p=q\leq P\text{ and }p\nmid m,
\end{cases}\]
so the denominator from the upper bound in~\eqref{eq:largersieve} becomes
\begin{align*}
\sum_{q\in\mathcal{S}}\frac{\Lambda(q)}{\#\Omega_q}-\log B
&=\log m+\sum_{\substack{c_0< p\leq P\\p\nmid m}}\frac{\log p}{p}\gcd(\ell,p-1)\left(1+O\left(\frac{1}{\sqrt{p}}\right)\right)-\log B\\
&=\log\left(\frac{m}{B}\right)
+\sum_{c_0< p\leq P}\frac{\log p}{p}\gcd(\ell,p-1)
-\sum_{p\mid m}\frac{\log p}{p}\gcd(\ell,p-1)+O(1).
\end{align*}
We can rewrite the second term as
\begin{align*}
\sum_{c_0< p\leq P}\frac{\log p}{p}\gcd(\ell,p-1)
&=\sum_{k\in(\Z/\ell\Z)^\times}\gcd(\ell,k-1)\sum_{\substack{c_0< p\leq P\\p\equiv k\bmod \ell}}\frac{\log p}{p}
\\
&=\frac{1}{\#(\Z/\ell\Z)^\times}\sum_{k\in(\Z/\ell\Z)^\times}\gcd(\ell,k-1)\cdot \log P+O(1)
\\
&=\tau(\ell)\cdot \log P+O(1),
\end{align*}
where the last equality follows from Menon's identity.
Then
\begin{equation}\label{eq:demlargersieve}
\sum_{q\in\mathcal{S}}\frac{\Lambda(q)}{\#\Omega_q}-\log B
=\log\left(\frac{m}{B}P^{\tau(\ell)}\right)+O(1)-\sum_{p\mid m}\frac{\log p}{p}\gcd(\ell,p-1).
\end{equation}
Taking
\[P\coloneqq c\left(\frac{B}{m}\right)^{\frac{1}{\tau(\ell)}}\exp\left(\frac{1}{\tau(\ell)}\sum_{\substack{p\mid m}}\frac{\log p}{p}\gcd(\ell,p-1)\right)\]
for some large enough constant $c$ ensures that~\eqref{eq:demlargersieve} is greater than $1$.
The bound from~\eqref{eq:largersieve} now becomes
\[\ll\sum_{q\in\mathcal{S}}\Lambda(q)-\log B
\leq \log m+\sum_{p\leq P}\log p-\log B\ll P,
\]
which is sufficient.
\end{proof}

For $F_{\mathbf{b}}(x)=F(x,\mathbf{b})$, define
\[\varrho_{\mathbf{b}}(m)\coloneqq \#\left\{x\in \Z/m\Z: F_{\mathbf{b}}(x)\equiv 0\bmod m\right\}.\]
The degree of $F_{\mathbf{b}}(x)$ is $d$ and the leading constant does not depend on $b$ because of the assumption~\ref{a:specialpoly}. Applying Lemma~\ref{lemma:largerapp} to $F_{\mathbf{b}}$ and $-F_{\mathbf{b}}$, we can bound~\eqref{eq:C4bound1} by
\begin{equation}\label{eq:C4bound2}
\sum_{\substack{n\in\Z\setminus\left\{0\right\}\\ p^{d-1}\mid n\text{ for some }p>T_1}}r(n,\B)\ll \sum_{1\leq m\leq \frac{X}{T_1^{d-1}}}\mu^2(m)\left(\frac{B_1}{m}\right)^{\frac{1}{\tau}}f(m)\sum_{\substack{\mathbf{b}\in\Z^{k-1}\\ |b_i|\leq B_i}}\varrho_{\mathbf{b}}(m),\end{equation}
where $\tau=\tau(d-1)$ and 
\begin{equation}\label{eq:deffmult}
f(m)\coloneqq\prod_{p\mid m}\left(1+\frac{2(d-1)}{\tau}\cdot\frac{\log p}{p}\right)
\end{equation}
is a multiplicative function.

\begin{lemma}\label{lemma:largeapp}
Assume that $F\in\Z[\x]$ satisfies~\ref{a:doubleroot} and~\ref{a:specialpoly}.
Suppose that $F$ has $r$ irreducible factors over $\Q$. 
Write
\[
m_{T_2}\coloneqq \prod_{p\leq T_2} p^{v_p(m)}.\]
Then for all $T_2,B_2,\dots,B_k,M\geq 2$, we have 
\[\sum_{\substack{1\leq m\leq M\\m_{T_2}\leq \sqrt{B}}}\mu^2(m)f(m)\sum_{\substack{\mathbf{b}\in\Z^{k-1}\\ |b_i|\leq B_i}}\varrho_{\mathbf{b}}(m)\ll M\frac{(\log M)^{d-1} }{(\log T_2)^{d-r}}\prod_{i=2}^k B_i,\]
where $B\coloneqq\min\{B_2,\dots, B_k\}$, $f$ is the multiplicative function defined in~\eqref{eq:deffmult}, and the implied constant depends only on $F$.
\end{lemma}
\begin{proof}
First fix a positive squarefree integer $m$ such that $m_{T_2}\leq \sqrt{B}$.

Fixing $\mathbf{b}$, we have $\rho_{\mathbf{b}}(p)\leq d$ as long as $p$ does not divide the coefficient of $x_1^d$ in $F$, so by~\ref{a:specialpoly} there are only finitely many exceptions to this upper bound. Then
\[\sum_{\substack{\mathbf{b}\in\Z^{k-1}\\ |b_i|\leq B_i}}\varrho_{\mathbf{b}}(m)\ll
d^{\#\{p>T_2:p\mid m\}}\sum_{\substack{\mathbf{b}\in\Z^{k-1}\\ |b_i|\leq B_i}}\varrho_{\mathbf{b}}(m_{T_2}).\]
Take $B'$ to be a multiple of $m_{T_2}$ such that $B'\geq B^2$. Then
\[\sum_{\substack{\mathbf{b}\in\Z^{k-1}\\ |b_i|\leq B_i}}\varrho_{\mathbf{b}}(m_{T_2})
\ll 
\frac{m_{T_2}}{B'}\#\left\{\x\in \Z^k: \begin{array}{l}
|x_1|\leq B',\ |x_i|\leq B_i\text{ for }2\leq i\leq k,\\ F(\x)\equiv 0\bmod m\end{array}\right\}.\]
To apply Lemma~\ref{lemma:largesieve}, define 
\[\Omega_p\coloneqq\begin{cases}
\left\{\x\in \F_p^{k}: F(\x)\equiv 0\bmod p\right\}&\text{if }p\mid m_{T_2},\\
\hfil \F_p^{k}&\text{otherwise}.
\end{cases}\]
Let
\[c_p\coloneqq \frac{1}{p^{k-1}}\#\left\{\x\in \F_p^k: F(\x)\equiv 0\bmod p\right\}.\]
We apply the large sieve in Lemma~\ref{lemma:largesieve} with
\[
L(\sqrt{B})=\sum_{\substack{1\leq t\leq \sqrt{B}\\ t\mid m_{T_2}}}\prod_{p\mid t}\frac{p^k-\#\Omega_p}{\#\Omega_p}
\geq
\sum_{t\mid m_{T_2}}\prod_{p\mid t}\frac{p^k-\#\Omega_p}{\#\Omega_p}
\geq
\prod_{p\mid m_{T_2}}\frac{p^k}{\#\Omega_p}
\geq \prod_{p\mid m_{T_2}}\frac{p}{c_p}.\]
We obtain
\begin{equation}\label{eq:sumrhob} 
\sum_{\substack{\mathbf{b}\in\Z^{k-1}\\ |b_i|\leq B_i}}\varrho_{\mathbf{b}}(m_{T_2})\ll \frac{m_{T_2}}{B'}B'\prod_{i=2}^k B_i\prod_{p\mid m_{T_2}}\frac{c_p}{p} 
=\prod_{i=2}^k B_i\prod_{p\mid m_{T_2}}c_p.
\end{equation}
The next step is to sum~\eqref{eq:sumrhob} over $m$ with the multiplicative factor $\mu^2(m)f(m)d^{\#\{p>T_2:p\mid m\}}$.
By~\eqref{eq:rhobounds} and~\ref{a:doubleroot}, $c_p$ is bounded by some constant depending only on $F$.
Lemma~\ref{lemma:cheb} and~\ref{a:doubleroot} imply that 
\[\sum_{p\leq T_2}\frac{c_p}{p}=r\log\log T_2+O(1)\quad\text{ and }\quad
\sum_{p\leq T_2}\frac{c_p\log p}{p^2}=O(1).\]
Therefore by~\cite{Wirsing}, we have
\begin{align*}
\sum_{\substack{1\leq m\leq M\\m_{T_2}\leq \sqrt{B}}}\mu^2(m)f(m)\sum_{\substack{\mathbf{b}\in\Z^{k-1}\\ |b_i|\leq B_i}}\varrho_{\mathbf{b}}(m)
&\ll
\prod_{i=2}^n B_i\cdot \sum_{\substack{1\leq m\leq M\\m_{T_2}\leq \sqrt{B}}}\mu^2(m)f(m)\prod_{p\mid m_{T_2}}c_p\prod_{\substack{p\mid m\\p>T_2}}d
\\
&\ll 
\frac{M}{\log M}\prod_{i=2}^n B_i\cdot \exp\left(\sum_{p\leq T_2}\frac{f(p)}{p}c_p+\sum_{T_2<p\leq M} \frac{f(p)}{p}d\right)\\
&\ll 
M\cdot\frac{(\log M)^{d-1} }{(\log T_2)^{d-r}}\cdot\prod_{i=2}^n B_i
\end{align*}
as claimed.
\end{proof}

\begin{lemma}\label{lemma:smoothapp}
Let $d>0$ and let $f$ be the multiplicative function defined in~\eqref{eq:deffmult}.
There exists some constant $c>0$ depending only on $d$ such that
\[\sum_{\substack{m\leq M\\m_{T_2}> \sqrt{B}}}\mu^2(m)f(m)d^{\omega(m)}\ll M(\log M)^{d-1}\exp\left(-c\frac{\log B}{\log T_2}\right)\]
for all $T_2,B,M\geq 2$.
The implied constant depends only on $d$.
\end{lemma}
\begin{proof}
Apply Lemma~\ref{lemma:T} to the multiplicative function $f(m)d^{\omega(m)}$. We get an upper bound
\[\ll \frac{M}{\log M}\exp\left(-c\frac{\log B}{\log T_2}+\sum_{p\leq M} \frac{d\cdot f(p)}{p}\right)\ll M(\log M)^{d-1}\exp\left(-c\frac{\log B}{\log T_2}\right).\]
\end{proof}

\begin{proof}[Proof of Lemma~\ref{lemma:C4}]
We will bound~\eqref{eq:C4bound2}. We use the fact that $B=\min\{B_2,\dots,B_k\}\geq Z$.
Take 
\[T_2\coloneqq\exp\left(\frac{c\log Z}{(d-r)\log\log Z}\right),\]
where $c$ is the constant from Lemma~\ref{lemma:smoothapp}.
Then it follows from Lemma~\ref{lemma:largeapp} and Lemma~\ref{lemma:smoothapp} that
\[\sum_{M<m\leq 2M}\mu^2(m)f(m)\left(\frac{B_1}{m}\right)^{\frac{1}{\tau}}\sum_{\substack{\mathbf{b}\in\Z^{k-1}\\ |b_i|\leq B_i}}\varrho_{\mathbf{b}}(m)
\ll \left(\frac{M}{B_1}\right)^{1-\frac{1}{\tau}}V(\log M)^{d-1}\left(\frac{\log\log Z}{\log Z}\right)^{d-r}.\]
Now splitting the sum~\eqref{eq:C4bound2} over dyadic intervals, we have 
\begin{multline*}
\sum_{1\leq m\leq \frac{X}{T_1^{d-1}}}\mu^2(m)f(m)\left(\frac{B_1}{m}\right)^{\frac{1}{\tau}}\sum_{\substack{\mathbf{b}\in\Z^{k-1}\\ |b_i|\leq B_i}}\varrho_{\mathbf{b}}(m)\\
\ll \left( \frac{X}{B_1T_1^{d-1}}\right)^{1-\frac{1}{\tau}}V\cdot (\log Y)^{d-1}\left(\frac{\log\log Z}{\log Z}\right)^{d-r}.\end{multline*}
Recall that $T_1=Y(\log Y)^{1-\delta}$. The assumptions~\eqref{eq:loglb} and~\eqref{eq:loggrowth} implies that
 \[\frac{X}{T_1^{d-1}}\ll \frac{B_1}{(\log Y)^{(1-\delta)(d-1)}}\ll \frac{B_1}{(\log Z)^{(1-\delta)(d-1)}}.\]
Therefore we conclude that
\begin{align*}
\sum_{\substack{n\in\Z\setminus\left\{0\right\}\\ p^{d-1}\mid n\text{ for some }p>T_1}}r(n,\B)&\ll
\frac{V}{(\log Z)^{(1-\delta)(d-1)(1-\frac{1}{\tau})}}(\log Y)^{d-1}\left(\frac{\log\log Z}{\log Z}\right)^{d-r}\\
&\ll
V\frac{\left(\log\log Z\right)^{d-r} }{(\log Z)^{(d-1)(\frac{d-r}{d-1}-\delta(1-\frac{1}{\tau})-\frac{1}{\tau})} },
\end{align*}
where we have applied the bound~\eqref{eq:loggrowth}.
The proof is complete on verifying that the choice of $\delta$ satisfies
$(d-1)(\frac{d-r}{d-1}-\delta(1-\frac{1}{\tau})-\frac{1}{\tau})=\delta$.
\end{proof}

\begin{proof}[Proof of Theorem~\ref{theorem:main1}]
Put together the error terms from Lemma~\ref{lemma:C1}, Lemma~\ref{lemma:smoothpart}, Lemma~\ref{lemma:C3} and Lemma~\ref{lemma:C4} with the choice of parameters given in~\ref{eq:parameters}.
\end{proof}

\section{Some consequences of Theorem~\texorpdfstring{\ref{theorem:Szpiroirr}}{1.2}}
To facilitate the proofs of Theorem~\ref{theorem:szpiroell} and Theorem~\ref{theorem:SzpiroTorsheight}, we establish two variants of Theorem~\ref{theorem:Szpiroirr} in this section, namely Theorem~\ref{theorem:main2} and Theorem~\ref{theorem:weightedhom}.

\begin{theorem}\label{theorem:main2}
Suppose $F\in\Z[\x]$ has total degree $d\geq 1$. Write \[F(\x)=c_0F_1(\x)^{e_1}F_2(\x)^{e_2}\dots F_{r}(\x)^{e_{r}},\]
 where $c_0\in\Q$, $F_i\in \Z[\x]$ are pairwise coprime, irreducible over $\Q$, and having total degrees $d_i\geq 1$. Let $B_1,\dots,B_k\geq 3$ such that~\eqref{eq:ABlog} holds.
Assume that for every $F_i$, there exists some $j$ such that $x_j^{d_i}$-coefficients is non-zero and 
\[
 B_1^{e_1}\dots B_k^{e_k}\ll (\max B_i)^{d_i-1}B_j
\]
for all integer $k$-tuple $(e_1,\dots,e_k)$ such that the $x_1^{e_1}\dots x_k^{e_k}$-coefficient of $F_i$ is non-zero.
 Let $w_i\coloneqq \max\left\{1,d_i-2\right\}$. Fix a constant $\beta>1$. Then we have
\[
 \frac{1}{B_1\dots B_k}\#\left\{\x\in\Z^k:|x_i|\leq B_i,\ |F_1(\x)|^{\frac{1}{w_1}}\cdots |F_{r}(\x)|^{\frac{1}{w_{r}}}
 \geq\rad(F(\x))^{\beta}\right\}
\ll \frac{(\log\log B_1)^2}{(\log B_1)^{\frac{1}{2}}},
\]
where
the implied constant depends at most on $\beta$, $F$ and the implied constants in~\eqref{eq:ABlog}.
\end{theorem}

Theorem~\ref{theorem:weightedhom} provides a result that is specific to binary forms.
To state Theorem~\ref{theorem:weightedhom}, we require the following definition.

\begin{definition}\label{def:Ffac}
Let $F(x,y)\in\Z[x,y]$ be a binary form of degree $d\geq 1$.
Write 
\[
F(x,y)=
c_0F_0(x,y)^{e_0}F_1(x,y)^{e_1}F_2(x,y)^{e_2}\dots F_{\psi}(x,y)^{e_{\psi}},
\]
 where $F_0(x,y)=y$, $e_0\geq 0$, $e_1,\dots,e_{\psi}\geq 1$, $c_0\in\Q$ and such that $F_i\in \Z[x,y]$ are pairwise coprime homogeneous polynomials that are irreducible over $\Q$. 
For a positive integer $m$, define 
\begin{equation}\label{eq:kappaF}
\kappa(m,F)\coloneqq\frac{md}{\frac{\delta_0}{w_0}+\frac{\delta_1}{w_1}+\dots +\frac{\delta_{\psi}}{w_{\psi}}},\end{equation}
and
\begin{equation}\label{eq:lambdaF}
\lambda(m,F)\coloneqq\frac{md}{\delta_0+\delta_1+\dots +\delta_{\psi}},\end{equation}
where
\[\delta_i\coloneqq \begin{cases}
 \hfil m\deg F_i &\text{ if }1\leq i\leq \psi,\\
\hfil \mathbf{1}_{e_0\geq 1}&\text{ if }i=0,
 \end{cases}\]
and
 \[w_i\coloneqq 
 \begin{cases}
\hfil \max\left\{1,\delta_i-2\right\}&\text{ if }x\nmid F_i(x,y), \\
\hfil 1&\text{ if }x\mid F_i(x,y).
\end{cases}\] 
\end{definition}

Define
\[\mathcal{B}(B_1,B_2)\coloneqq\left\{(a,b)\in\Z^2: |a|\leq B_1,\ |b|\leq B_2\right\}.\]

\begin{theorem}\label{theorem:weightedhom}
Let $F\in\Z[x,y]$ be a binary form of degree $d\geq 1$.
Fix a constant $\beta>\kappa(m,F)$, where $\kappa$ is defined in~\eqref{eq:kappaF}.
Then for all $X\geq 3$ we have
\[
 \frac{1}{X^{\frac{m+1}{md}}}\#\left\{(a,b)\in\mathcal{B}\left(X^{\frac{1}{d}},X^{\frac{1}{md}}\right): \rad(F(a,b^m))^{\beta}\leq X\right\}\ll \frac{(\log\log X)^2}{(\log X)^{\frac{1}{2}}},\]
 where the implied constant depends at most on $\beta$, $F$, and $m$.
 Moreover
\[\rad(F(a,b^m))^{\lambda(m,F)}\ll X\quad \text{for all}\quad(a,b)\in\mathcal{B}\left(X^{\frac{1}{d}},X^{\frac{1}{md}}\right),\]
where $\lambda$ is defined in~\eqref{eq:lambdaF} and the implied constant depends only on $F$ and $m$.
 \end{theorem}

\subsection{Common prime factors of polynomials}
We will prove Theorem~\ref{theorem:main2} through the next two lemmas.
As before, write
\[V\coloneqq B_1\dots B_k.\]
\begin{lemma}\label{lemma:hompolygcd}
Let $F_1,F_2\in \Z[\x]$ be coprime in $\Q[\x]$. 
Then for $B_1,\dots,B_k,X\geq 2$, we have
\[\#\left\{\x\in\Z^k:|x_i|\leq B_i,\ \rad\left(\gcd(F_1(\x),F_2(\x))\right)\geq X\right\}\ll \frac{V}{\log X}+\frac{V\log V}{\min B_i},\]
where the implied constant depends only on $F_1$ and $F_2$.
\end{lemma}
\begin{proof}
Let $Z=\min B_i$.
By Lemma~\ref{lemma:geoms},
we can assume that any $p\mid \gcd(F_1(\x),F_2(\x))$ is less than $Z$ with 
\[\ll\frac{V\log V}{Z}\]
many exceptions.

Let 
\[r(g,\B)\coloneqq \#\left\{\x\in\Z^k:|x_i|\leq B_i,\ \rad\left(\gcd(F_1(\x),F_2(\x))\right)=g\right\}\]
and 
\[T(p,\B)\coloneqq \#\left\{\x\in\Z^k:|x_i|\leq B_i,\ F_1(\x)\equiv F_2(\x)\equiv 0\bmod p\right\}.\]
Suppose $g=\rad\left(\gcd(F_1(\x),F_2(\x))\right)$.
Since $F_1, F_2$ are coprime, for sufficiently large $p$, 
\[T(p,\B)\ll \prod_{i=1}^k\left(\frac{B_i}{p^2}+1\right)\ll \frac{V}{p^2}\]
as long as $p\leq Z$.
Then
\begin{equation}\label{eq:prerg}
\sum_{\substack{g\geq X\\ p\mid g\Rightarrow p\leq Z}}r(g,\B)\leq \sum_{\substack{g\\ p\mid g\Rightarrow p\leq Z}}r(g,\B)\frac{\log g}{\log X}.\end{equation}
We have the bound
\[\sum_{\substack{g\\ p\mid g\Rightarrow p\leq Z}}r(g,\B)\log g=\sum_{p\leq Z}T(p,\B)\log p\ll V\sum_{p\leq Z}\frac{\log p}{p^2}\ll V.\]
Putting this back into~\eqref{eq:prerg} gives the required upper bound.
\end{proof}

\begin{lemma}\label{lemma:rehompolygcd}
Let $F_1,F_2,\dots,F_{r}\in \Z[\x]$ pairwise coprime in $\Q[\x]$. 
Then for $B_1,\dots, B_k,X\geq 2$, we have
\begin{multline*}
\#\left\{\x\in\Z^k:|x_i|\leq B_i,\ \rad(F_1(\x))\cdots\rad(F_r(\x))\geq X\rad(F_1(\x)\dots F_r(\x))\right\}\\
\ll \frac{V}{\log X}+\frac{V\log V}{\min B_i},
\end{multline*}
where the implied constant depends only $F_1,\dots,F_{r}$.
\end{lemma}
\begin{proof}
Observe that
\[\rad(F_1(\x))\cdots\rad(F_r(\x))\leq \rad(F_1(\x)\dots F_r(\x))\prod_{\substack{i,j\\ i\neq j}}\gcd(\rad(F_i(\x)),\rad(F_j(\x))).\]
By Lemma~\ref{lemma:hompolygcd}, $\gcd(\rad(F_i(\x)),\rad(F_j(\x)))\leq X^{\frac{1}{r^2}}$ holds for all $i\neq j$ with at most $O( \frac{V}{\log X}+\frac{V\log V}{\min B_i})$ exceptions.
\end{proof}

\begin{proof}[Proof of Theorem~\ref{theorem:main2}]
Apply Theorem~\ref{theorem:Szpiroirr} to each of $F_1,\dots F_{r}$.
Therefore we can assume that $|F_i(\x)|<\rad(F_i(\x))^{\beta_i}$ for all $i$ by fixing some $\beta_i>w_i$.
Then
\[|F_1(\x)|^{\frac{1}{\beta_1}}|F_2(\x)|^{\frac{1}{\beta_2}}\dots |F_{r}(\x)|^{\frac{1}{\beta_{r}}}<\rad(F_1(\x))\rad(F_2(\x))\cdots \rad(F_{r}(\x)).\]
Let $Z\coloneqq \min B_i$.
By Lemma~\ref{lemma:rehompolygcd}, we can assume that
\[\rad(F_1(\x))\rad(F_2(\x))\cdots \rad(F_{r}(\x))\leq \rad(F(\x)) Z^{\epsilon},\]
where $\epsilon>0$.
By Lemma~\ref{lemma:C1}, we can assume $\frac{Z}{\log Z}< |F_i(\x)|$.
Pick $\epsilon$ small enough so that $\frac{1}{w_i\beta}<\frac{1}{\beta_i}-\frac{2}{r}\epsilon$ for all $i$, then
\[|F_1(\x)|^{\frac{1}{w_1\beta}}|F_2(\x)|^{\frac{1}{w_2\beta}}\dots |F_{r}(\x)|^{\frac{1}{w_{r}\beta}}< \rad(F(\x)).\]
The number of exceptions is acceptable noting the assumption~\eqref{eq:ABlog}.
\end{proof}

\subsection{Binary form estimates}
To prove Theorem~\ref{theorem:weightedhom}, we have to utilise the fact that $F$ is a binary form.
\begin{lemma}\label{lemma:binaryformarea}
Suppose $F(x,y)\in \Z[x,y]$ is a binary form of degree $d\geq 1$ with non-zero discriminant. Let $m$ be a positive integer and $0<\epsilon<\frac{1}{m}$. Then for all $Y\geq 2$, we have
\[\#\left\{
(a,b)\in\mathcal{B}\left(Y,Y^{\frac{1}{m}}\right):
 0<|F(a,b^m)|\leq Y^{d-\epsilon}\right\}
\ll Y^{1+\frac{1}{m}-\frac{1}{md}\epsilon},\]
where the implied constant depends at most on $F$, $m$ and $\epsilon$.
\end{lemma}
\begin{proof}

Let 
 \[\delta\coloneqq \begin{cases}
\hfil \frac{1}{m}\left(1-\frac{1}{d-1}\epsilon\right)&\text{if }d\neq 1,\\
\hfil \frac{1}{m}-\epsilon&\text{if }d=1.
\end{cases}
\]
Trivially
\[\#\left\{(a,b)\in\Z^2: |a|\leq Y,\ |b|\leq Y^{\delta}\right\}\leq Y^{1+\delta},\]
so it suffices to bound
\begin{equation}\label{eq:gensetbinary}
\#\left\{(a,b)\in \mathcal{B}\left(Y,Y^{\frac{1}{m}}\right): |b|>Y^{\delta},\ 0<|F(a,b^m)|\leq Y^{d-\epsilon}\right\}.
\end{equation}

Let $f(t)=F(t,1)$, so $F(a,b^m)=b^{md}f(\frac{a}{b^m})$ and $0<|F(a,b^m)|\leq Y^{d-\epsilon}$ implies
\[0<\left|f\left(\frac{a}{b^m}\right)\right|\leq \frac{Y^{d-\epsilon}}{|b|^{md}},\]
when $b\neq 0$.
By hypothesis we have that the degree of $f$ is $d$ or $d-1$, and the discriminant of $f$ is non-zero. The degree of $f$ is at least $1$ unless $F(x,y)=cy$ for some integer $c$. If $F(x,y)=cy$, there are $O(Y)$-many possible $a$ and $|b^m|\ll Y^{1-\epsilon}$, so we have an upper bound of $O(Y^{1+\frac{1}{m}(1-\epsilon)})$ as required. Therefore let us assume that $\deg f\geq 1$.

Let $\xi_1,\dots,\xi_r$ be the real roots of $f$. By~\cite[Lemma~2.3(ii)]{FouvryWaldschmidt}, there exists some $i$ such that 
 \[0<\left|\frac{a}{b^m}-\xi_i\right|\ll\frac{Y^{d-\epsilon}}{|b|^{md}},\]
 equivalently
 \[0<\left|a-\xi_i b^m\right|\ll\frac{Y^{d-\epsilon}}{|b|^{m(d-1)}}.\]
 Given $b$, there are
 \[O\left(\frac{Y^{d-\epsilon}}{|b|^{m(d-1)}}+1\right)\]
 many possible $a$.

 Now sum over $Y^{\delta}<|b|\leq Y^{\frac{1}{m}}$. We have
 \[\sum_{Y^{\delta}<|b|\leq Y^{\frac{1}{m}}} \left(\frac{Y^{d-\epsilon}}{|b|^{m(d-1)}}+1\right)
\ll \begin{cases}
\hfil Y^{d-\epsilon-\delta(m(d-1)-1)}+Y^{\frac{1}{m}}&\text{if }m(d-1)\geq 2,\\
\hfil Y^{2-\epsilon}\log Y+Y&\text{if }d=2\text{ and }m=1,\\
\hfil Y^{1+\frac{1}{m}-\epsilon}+Y^{\frac{1}{m}}&\text{if }d=1.
\end{cases}
\]
This allows us to bound~\eqref{eq:gensetbinary} by 
 \[\ll \begin{cases}
\hfil Y^{1+\delta}&\text{if }m(d-1)\geq 3\text{ or }d=1,\\
\hfil Y^{2-\epsilon}\log Y+Y\ll Y^{2-\frac{1}{2}\epsilon}&\text{if }d=2\text{ and }m=1,\end{cases}
\]
which is sufficient.
\end{proof}

We remove the non-zero discriminant by applying Lemma~\ref{lemma:binaryformarea} to each irreducible factor of $F(x,y)$. It will also be convenient for our later applications to drop the contribution from the primes $2$ and $3$ to the size of $F(x,y)$.
\begin{lemma}\label{lemma:rebinaryformarea}
Suppose $F(x,y)\in \Z[x,y]$ is a binary form of degree $d\geq 1$. Let $m$ be a positive integer and $0<\epsilon<\frac{1}{m}$. Then
\[\#\left\{
(a,b)\in\mathcal{B}\left(Y,Y^{\frac{1}{m}}\right):
 0<\prod_{p\neq 2,3} p^{v_p(F(a,b^m))}\leq Y^{d-\epsilon}\right\}
\ll
 \frac{Y^{1+\frac{1}{m}}}{\log Y},
\]
where the implied constant depends at most on $F$, $m$ and $\epsilon$.
\end{lemma}
\begin{proof}
Write
\[F(x,y)=c_0F_1(x,y)\dots F_{\psi}(x,y),\]
where $F_1,\dots,F_{\psi}\in\Z[x,y]$ are irreducible polynomials with positive degrees (not necessarily distinct), and $c_0\in\Q$. Let $d_i$ be the degree of $F_i$.

 Apply Lemma~\ref{lemma:smoothpart} to each $F(a,b^m)$ with $t=3$, then we can assume that 
\[2^{v_2(F_i(a,b^m))}3^{v_3(F_i(a,b^m))}\leq Y^{\frac{1}{2\psi}\epsilon}\]
with the number of exceptions bounded by 
\[\ll \frac{Y^{1+\frac{1}{m}}}{\log Y}.\]

If $0<|F(a,b^m)|\ll Y^{d-\frac{1}{2}\epsilon}$, then there must be some $i$ such that 
$0<|F(a,b^m)|\ll Y^{d_i-\frac{1}{2\psi}\epsilon}$.
Apply Lemma~\ref{lemma:binaryformarea} to each $F_i$, we have an upper bound
\[\#\left\{
(a,b)\in\mathcal{B}\left(Y,Y^{\frac{1}{m}}\right):
 0<|F_i(a,b^m)|\leq Y^{d_i-\frac{1}{2\psi}\epsilon}\right\}\ll Y^{1+\frac{1}{m}-\frac{1}{2md_i\psi}\epsilon}.\]
 Summing this upper bound over $i$ completes the proof.
\end{proof}

\begin{proof}[Proof of Theorem~\ref{theorem:weightedhom}]
Take $\epsilon>0$ small enough such that there exists $\beta_0$ satisfying $(\frac{1}{\kappa(m,F)}-\epsilon)\beta>\beta_0>1$. By Theorem~\ref{theorem:main2}, we may assume that 
\[\rad(F(a,b^m))^{\beta_0}\gg |b|^{\frac{1}{w_0}\mathbf{1}_{e_0\geq 1}}|F_1(a,b^m)|^{\frac{1}{w_1}}\cdots |F_{\psi}(a,b^m)|^{\frac{1}{w_{\psi}}}.\]
Apply Lemma~\ref{lemma:binaryformarea} with $Y=X^{\frac{1}{d}}$, we may assume that $|F_i(a,b^m)|>X^{\frac{\deg F_i}{d}-\frac{1}{\psi+1}\epsilon}$.
Then
\[\rad(F(a,b^m))^{\beta_0}\gg X^{\frac{1}{\kappa(m,F)}-\epsilon},\]
which implies that $\rad(F(a,b^m))^{\beta}> X$ for large enough $X$ as required.

For the final claim, 
\[\rad(F(a,b^m))\ll |F_0(a,b^m)|^{\mathbf{1}_{e_0\geq 1}}|F_1(a,b^m)|\cdots |F_{\psi}(a,b^m)|\ll X^{\frac{1}{md}(\delta_0+\dots+\delta_\psi)},\]
where $\delta_0,\dots,\delta_\psi$ are as in Definition~\ref{def:Ffac}.
\end{proof}

\section{Szpiro ratio of one-parameter families}\label{section:oneparfam}
In this section, we work towards a proof of Theorem~\ref{theorem:szpiroell}.
Suppose $E_{A,B}:y^2=x^3+Ax+B$ is a minimal short Weierstrass model of an elliptic curve over $\Q$, so $A,B\in\Z$, and there exists no prime $p$ such that $p^4\mid A$ and $p^6\mid B$.
The discriminant of $E_{A,B}$ is
\[\Disc(E_{A,B})=-16(4A^3+27B^2),\]
which is minimal at primes $p\neq 2,3$, so
\[\prod_{p\neq 2,3}p^{v_p(\Disc(E_{A,B})) }\leq |\Disc_{\min}(E_{A,B})|\leq |\Disc(E_{A,B})|.\]
 By Tate's algorithm~\cite{Tate}, the conductor of $E_{A,B}$ is 
\[N(E_{A,B})=\prod_p p^{f_p},\]
where
\[f_p\begin{cases}
\hfil =0&\text{ if }p\nmid \Disc(E_{A,B}),\\
\hfil =1&\text{ if }p\mid \Disc(E_{A,B})\text{ and }p\nmid A\text{ and }p\neq 2,3,\\
\hfil =2&\text{ if }p\mid \Disc(E_{A,B})\text{ and }p\mid A\text{ and }p\neq 2,3,\\
\hfil \leq 8&\text{ if }p=2,3.
\end{cases}
\]
Therefore
\begin{equation}\label{eq:ineqNrad}
\rad( \Disc(E_{A,B}))\ll \prod_{\substack{p\mid \Disc(E_{A,B})\\p\neq 2,3}}p\leq N(E_{A,B})
\ll \rad(\gcd(A,B))\cdot \rad( \Disc(E_{A,B})),
\end{equation}
where the implied constants are absolute.

\begin{definition}
Let $\nu\in\left\{1,2\right\}$.
Define $\mathcal{P}_\nu$ to be the set of $(f,g)$, where $f,g\in\Q[t]$ are coprime polynomials such that
\begin{itemize}
 \item there exists positive integers $n,m$ such that 
 \begin{equation}\label{eq:mndef}
 \nu\cdot \max\left\{\frac{1}{4}\deg f,\frac{1}{6}\deg g\right\}=\frac{n}{m}
 \end{equation}
with $m=1$ or $n=1$,
\item $n=1$ if $\nu=2$.
\end{itemize}
\end{definition}

We are interested the Szpiro ratio in families of elliptic curves of the form 
\[\mathcal{E}_t:y^2=x^3+f(t)x+g(t),\ t\in\Q,\]
where $(f,g)\in\mathcal{P}_1$ or $\mathcal{P}_2$.
Given $t\in \Q$, write $t=a/b^m$ for some integers $a,b$ such that $\gcd(a,b^m)$ is not divisible by any $m$-th power.
Define the \emph{parameter height} of $\mathcal{E}_t$ by $\max\{|a|,|b^m|\}$.
In this section, we will study the Szpiro ratio of such families ordered by the parameter height and by the naive height (defined in 
\eqref{eq:naiveheight}).

We first define some constants depending on $(f,g)$, which will serve as upper and lower bounds for the Szpiro ratio in our results.

\begin{definition}
Suppose $(f,g)\in\mathcal{P}_\nu$ for some $\nu\in\left\{1,2\right\}$.
Let $m, n$ be the integers that satisfy~\eqref{eq:mndef}. Let $d(t)\coloneqq 4f(t)^3+27g(t)^2$ and
 \[D(x,y)\coloneqq y^{\frac{12n}{\nu m}}d\left(\frac{x}{y^m}\right).\]
Define 
\begin{equation}\label{eq:lambda}
\lambda_{\nu}(f,g)\coloneqq\lambda(m,D)=\frac{12n/\nu}{\delta_0+\delta_1+\dots +\delta_{\psi}}\end{equation}
and
\begin{equation}\label{eq:kappa}
\kappa_{\nu}(f,g)\coloneqq\kappa(m,D)=\frac{12n/\nu}{\frac{\delta_0}{w_0}+\frac{\delta_1}{w_1}+\dots +\frac{\delta_{\psi}}{w_{\psi}}},\end{equation}
where $\lambda$, $\kappa$, $\delta_i$ and $w_i$ are as in Definition~\ref{def:Ffac}.
\end{definition}

\subsection{Ordering by parameter height}
We show that $\sigma(\mathcal{E}_t)$ is bounded above for almost all $t$, when ordered by its parameter height.
\begin{theorem}\label{theorem:parameter height}
Let $(f,g)\in\mathcal{P}_1$. Let $m, n$ be the integers that satisfy~\eqref{eq:mndef}. Fix $\beta> \kappa_{1}(f,g)$. For any $H\geq 3$, we have 
\[
\frac{1}{H^{\frac{m+1}{12n}}}\#\left\{\frac{a}{b^m}\in \Q: (a,b)\in\mathcal{B}\left(H^{\frac{ m}{12n}}, H^{\frac{1}{12n}}\right),\ \sigma(\E_{\frac{a}{b^m}})\geq\beta\right\}\ll \frac{(\log\log H)^2}{(\log H)^{\frac{1}{2}}},
\]
where the implied constants depend on $f,g$ and $\beta$.
\end{theorem}
We will deduce Theorem~\ref{theorem:parameter height} from the following lemma.
\begin{lemma}\label{lemma:parameter height}
Let $(f,g)\in\mathcal{P}_\nu$, $\nu\in\left\{1,2\right\}$. Let $m, n$ be the integers that satisfy~\eqref{eq:mndef}. Fix $\beta> \kappa_{\nu}(f,g)$. Given any $a,b,c\in\Z$, let
$(A,B)\coloneqq \left(b^{\frac{4n}{\nu}}c^2f(\frac{a}{b^m}),b^{\frac{6n}{\nu}}c^3g(\frac{a}{b^m})\right)$.
There exists some $\epsilon>0$ depending only on $\beta$ such that for any $H\geq 3$ and uniformly for $1\leq c\leq H^{\epsilon}$, we have
\[
 \frac{1}{H^{\frac{\nu(m+1)}{12n}}}\#\left\{(a,b)\in\mathcal{B}\left(H^{\frac{\nu m}{12n}}, H^{\frac{\nu}{12n}}\right): \sigma(E_{A,B})\geq\beta\right\}\ll \frac{(\log\log H)^2}{(\log H)^{\frac{1}{2}}},\]
where the implied constants depend on $f,g$ and $\beta$.
\end{lemma}

\begin{proof}
Let $F(a,b^m)=b^{\frac{4n}{\nu}}f(a/b^m)$, $G(a,b^m)=b^{\frac{6n}{\nu}}g(a/b^m)$, and $D(a,b^m)=b^{\frac{12n}{\nu}}(4f(a/b^m)^3+27g(a/b^m)^2)$.
Then
\[\Disc_{\min}(E_{A,B})\ll c^6D(a,b^m)\leq H^{6\epsilon}D(a,b^m).\]
Applying Lemma~\ref{lemma:rebinaryformarea} to $D$, we may assume that $ H^{1-\epsilon}\ll |D(a,b^m)|\ll H$,
so
\[\Disc_{\min}(E_{A,B})\ll H^{1+6\epsilon}.\]

As for the conductor, we see from~\eqref{eq:ineqNrad} that
\[N(E_{A,B})\gg\frac{\rad(D(a,b^m))}{\rad(\gcd(F(a,b^m),G(a,b^m))}.\]
Lemma~\ref{lemma:hompolygcd} allows us to assume that $\rad(\gcd(F(a,b^m),G(a,b^m))\leq H^{\epsilon}$.
Taking $\epsilon>0$ is small enough, the claim follows from applying Theorem~\ref{theorem:weightedhom} with $F=D$, $d=\frac{12n}{\nu m}$, and $X=H$.
\end{proof}

\begin{proof}[Proof of Theorem~\ref{theorem:parameter height}]
Put $\nu=1$ and $c=1$ into Lemma~\ref{lemma:parameter height} and note that the curve $\mathcal{E}_t:y^2=x^3+f(t)x+g(t)$ with $t=\frac{a}{b^m}$ is $\Q$-isomorphic to
$y^2=x^3+F(a,b^m)x+G(a,b^m)$.
\end{proof}

\subsection{Ordering by naive height}
Now we will use the setup in~\cite{HarronSnowden} to convert the ordering from parameter height to naive height.
\begin{definition}
Suppose $(f,g)\in\mathcal{P}_\nu$ for some $\nu\in\left\{1,2\right\}$.
Let $S(H)\coloneqq S_{f,g,\nu}(H)$ be the set of pairs $(A,B)\in\Z^2$ satisfying all of the following conditions:
\begin{itemize}
\item $4A^3+27B^2\neq 0$,
\item $\gcd(A^3,B^2)$ is not divisible by any $12$th power,
\item $|A|\leq (\frac{1}{4}H)^{\frac{1}{3}}$ and $|B|\leq (\frac{1}{27}H)^{\frac{1}{2}}$,
\item there exists $u,t\in\Q$ such that $A=u^{\frac{4}{\nu}}f(t)$ and $B=u^{\frac{6}{\nu}}g(t)$.
\end{itemize}
\end{definition}

Theorem~\ref{theorem:szpiroell} is a special case of Theorem~\ref{theorem:szpiro1}.
\begin{theorem}\label{theorem:szpiro1}
Let $(f,g)\in\mathcal{P}_\nu$, $\nu\in\left\{1,2\right\}$. Let $m, n$ be the integers that satisfy~\eqref{eq:mndef}.
Fix constants $\beta_1<\lambda_{\nu}(f,g)\leq \kappa_{\nu}(f,g)<\beta_2$.
Then for all $H\geq 3$ we have
\[\frac{\#\left\{(A,B)\in S(H): \sigma(E_{A,B})\notin (\beta_1,\beta_2)\right\}}{\#S(H)}\ll \frac{(\log\log H)^2}{(\log H)^{\frac{1}{2}}},\]
where the implied constant depends on $(f,g)$ and $(\beta_1,\beta_2)$.
\end{theorem}

If $(A,B)\in S(H)$, then $E_{A,B}:y^2=x^3+Ax+B$ is a minimal short Weierstrass equation with height at most $H$.
The discriminant of $E_{A,B}$ is
\[\Disc(E_{A,B})=-16(4A^3+27B^2)=-16u^{12}d(t).\]

We collect some results needed for the proof of Theorem~\ref{theorem:szpiro1}.
The first is the order of $S(H)$ obtained in~\cite{HarronSnowden}.
\begin{theorem}[{\cite[Theorem~1.7 and Proposition~4.1]{HarronSnowden}}]\label{theorem:HS}
Let $(f,g)\in\mathcal{P}_\nu$, $\nu\in\left\{1,2\right\}$. Let $m, n$ be the integers that satisfy~\eqref{eq:mndef}.
Then 
\[\# S(H)\asymp H^{\frac{\nu(m+1)}{12n}}.\]
\end{theorem}

The following lemma combines~\cite[Lemma~2.5 and Lemma~2.6]{HarronSnowden} and their analogues used in the proof of~\cite[Proposition~4.1]{HarronSnowden}.
\begin{lemma}[{\cite{HarronSnowden}}]\label{lemma:ratcon}
Let $(f,g)\in\mathcal{P}_\nu$, $\nu\in\left\{1,2\right\}$. Let $m, n$ be the integers that satisfy~\eqref{eq:mndef}.
There exists a finite set $Q$ of non-zero rational numbers such that the following holds.
Suppose $(u,t)\in\Q^2$ such that $\left(u^{\frac{4}{\nu}}f(t),u^{\frac{6}{\nu}}g(t)\right)\in S(H)$. Write $t=\frac{a}{b^m}$, where $a,b\in\Z$ such that $b>0$ and $\gcd(a,b^m)$ is not divisible by any $m$-th power.
Then
\[u=qcb^n,\text{ where }q\in Q\text{ and }
\begin{cases}
\hfil c=1&\text{if }\nu=1,\\
\hfil c\text{ is a positive squarefree integer}&\text{if }\nu=2.\end{cases}\]
Moreover
\begin{equation}\label{eq:rangeab}
a\ll (H/c^6)^{\frac{\nu m}{12n}}\qquad\text{and}\qquad b\ll (H/c^6)^{\frac{\nu }{12n}},
\end{equation}
where the implied constants depend only on $f$ and $g$.
\end{lemma}

We are now ready to prove Theorem~\ref{theorem:szpiro1}.
\begin{proof}[Proof of Theorem~\ref{theorem:szpiro1}]
Theorem~\ref{theorem:HS} gives $\#S(H)\gg H^{\frac{\nu(m+1)}{12n}}$, so it suffices to bound
$\#\left\{(A,B)\in S(H): \sigma(E_{A,B})\notin (\beta_1, \beta_2)\right\}$.

Given a positive integer $c$, take
\[\tilde{H}\asymp H/c^6
\] such that any $(a,b)$ satisfying~\eqref{eq:rangeab} lies in $\mathcal{B}(\tilde{H}^{\frac{\nu m}{12n}}, \tilde{H}^{\frac{\nu}{12n}})$.
If $\nu=1$, then $c=1$.
If $\nu=2$, by assumption $n=1$, so
\[\sum_{c> H^{\frac{1}{6}\epsilon}}\#\mathcal{B}(\tilde{H}^{\frac{ m}{6}}, \tilde{H}^{\frac{1}{6}})\ll \sum_{c> H^{\frac{1}{6}\epsilon}} (H^{\frac{1}{6}\epsilon}/c)^{m+1}\ll H^{\frac{m(1-\epsilon)+1}{6}}.\]
Therefore we may assume $c\leq H^{\frac{1}{6}\epsilon}$. Since $\tilde{H}\gg H/c^6\geq H^{1-\epsilon}$, we also have $c\ll \tilde{H}^{\frac{\epsilon}{6(1-\epsilon)}}$.

First we bound the exceptions to $\sigma(E_{A,B})\geq\beta_2$.
 By Lemma~\ref{lemma:ratcon}, if $(A,B)\in S(H)$, then 
\[
A=(qcb)^{\frac{4}{\nu}}f(a/b^m)\qquad\text{and}\qquad B=(qcb)^{\frac{6}{\nu}}g(a/b^m),
\]
where $a,b$ satisfy~\eqref{eq:rangeab}, $q\neq 0$ lies in some finite set, and $c$ is squarefree.
By Lemma~\ref{lemma:parameter height}, choosing $\epsilon$ is small enough, we have uniformly for all $1\leq c\leq H^{\frac{1}{6}\epsilon}$, the bound
\begin{equation}\label{eq:beta2}
\#\left\{(a,b)\in \mathcal{B}(\tilde{H}^{\frac{\nu m}{12n}}, \tilde{H}^{\frac{\nu}{12n}}): \sigma(E_{A,B})\geq\beta_2\right\}\ll \tilde{H}^{\frac{\nu(m+1)}{12n}}\frac{(\log\log \tilde{H})^2}{(\log \tilde{H})^{\frac{1}{2}}}.
\end{equation}

Next we bound the exceptions to $\sigma(E_{A,B})\leq\beta_1$.
Let $F(a,b^m)=(qb^n)^{\frac{4}{\nu}}f(a/b^m)$, $G(a,b^m)=(qb^n)^{\frac{6}{\nu}}g(a/b^m)$, and $D(a,b^m)=4F(a,b^m)^3+27G(a,b^m)^2$, so when $(A,B)\in S(H)$, we have
\[\Disc(E_{A,B})=-16c^6D(a,b^m)
\]
We can assume by Lemma~\ref{lemma:hompolygcd} that $\rad(\gcd(A,B))<\tilde{H}^{\epsilon}|c|\ll \tilde{H}^{\epsilon(1+\frac{1}{6(1-\epsilon)})}$ since $F$ and $G$ are coprime polynomials over $\Q$, and $|c|\ll \tilde{H}^{\frac{\epsilon}{6(1-\epsilon)}}$.
By~\eqref{eq:ineqNrad}, the conductor is bounded by
\begin{equation}\label{eq:ineqND}
N(E_{A,B})\ll \rad(\gcd(A,B))\cdot \rad(\Disc(E_{A,B}))
 \ll \tilde{H}^{\epsilon(1+\frac{1}{3(1-\epsilon)})}\cdot \rad(D(a,b^m)), 
 \end{equation}
under the above assumptions.
Then from~\eqref{eq:ineqND} and Theorem~\ref{theorem:weightedhom}, we have
\[ \tilde{H}^{-\epsilon(1+\frac{1}{3(1-\epsilon)})}N(E_{A,B})\ll \rad(D(a,b^m))\ll \tilde{H}^{\frac{1}{\lambda_{\nu}(f,g)}}.\]
Applying Lemma~\ref{lemma:rebinaryformarea} to $D(a,b^m)$, we may assume that
\[\Disc_{\min}(E_{A,B})\gg \prod_{p\neq 2,3}p^{v_p(D(a,b^m))}\gg \tilde{H}^{1-\epsilon},\]
as long as $(A,B)\in S(H)$.
Therefore taking $\epsilon$ small enough such that $\frac{1}{\lambda_{\nu}(f,g)}+\epsilon(1+\frac{1}{3(1-\epsilon)})<\frac{1-\epsilon}{\beta_1}$, then putting together all the error terms,
we see that the number of exceptions to 
$\sigma(E_{A,B})>\beta_1$, uniformly for all $1\leq c\leq H^{\frac{1}{6}\epsilon}$, is bounded by
\begin{equation}\label{eq:beta1}
\#\left\{(a,b)\in \mathcal{B}(\tilde{H}^{\frac{\nu m}{12n}}, \tilde{H}^{\frac{\nu}{12n}}): \sigma(E_{A,B})\leq\beta_1,\ (A,B)\in S(H)\right\}\ll\frac{\tilde{H}^{\frac{\nu(m+1)}{12n}}}{\log \tilde{H}}.\end{equation}

Combining~\eqref{eq:beta2} and~\eqref{eq:beta1}, we have
\[ \#\left\{(a,b)\in \mathcal{B}(\tilde{H}^{\frac{\nu m}{12n}}, \tilde{H}^{\frac{\nu}{12n}}): \sigma(E_{A,B})\notin(\beta_1,\beta_2),\ (A,B)\in S(H)\right\}\ll \tilde{H}^{\frac{\nu(m+1)}{12n}}\frac{(\log\log \tilde{H})^2}{(\log \tilde{H})^{\frac{1}{2}}}.\]
When $\nu=1$, $c=1$ so we are done once we put in $\tilde{H}\ll H$.
When $\nu=2$, we sum the bound over $c$ to get
\[\#\left\{(A,B)\in S(H): \sigma(E_{A,B})\notin(\beta_1,\beta_2)\right\}\ll 
\sum_{1\leq c\leq H^{\frac{1}{6}\epsilon}}\left(\frac{H}{c^6}\right)^{\frac{m+1}{6}}\frac{(\log\log (H/c^6))^2}{(\log (H/c^6))^{\frac{1}{2}}}.\]
Since $m\geq 1$, the exponent of $c$ is $-(m+1)\leq -2$, so the sum over $c$ in the case $\nu=2$ converges and we obtain the required upper bound.
\end{proof}

\begin{proof}[Proof of Theorem~\ref{theorem:szpiroell}]
Theorem~\ref{theorem:szpiroell} is a corollary of Theorem~\ref{theorem:szpiro1} with $\nu=1$.
\end{proof}

\section{Szpiro ratio in families with prescribed torsion}
The goal of this section is to prove Theorem~\ref{theorem:SzpiroTorsheight} using Theorem~\ref{theorem:szpiro1}.
The possible torsion subgroups of elliptic curves over $\Q$ were classified by Mazur.
\begin{theorem}[{Mazur~\cite{Mazur,Mazur2}}]
Let $E$ be a rational elliptic
curve. Then $E(\Q)_{\mathrm{tors}}$ is isomorphic to one of the following groups:
\begin{equation}\label{eq:torgroup}
\begin{cases}
\hfil C_{n} & \text{for }n=1,2,\dots,10,12,\\
C_{2}\times C_{2n} & \text{for }n=1,2,3,4.
\end{cases}
\end{equation}
\end{theorem}

For any $G$ permitted by~\eqref{eq:torgroup}, we are interested in the family of elliptic curves $E/\Q$ with torsion subgroup $E(\Q)_{\mathrm{tors}}\cong G$. We treat the following four cases separately:
\begin{itemize}
\item $G\neq C_3,\ C_2,\ C_2\times C_2$,
\item $G= C_3$,
\item $G= C_2\times C_2$,
\item $G= C_2$.
\end{itemize}

\subsection{When \texorpdfstring{$G\neq C_3,\ C_2,\ C_2\times C_2$}{G is not C\_3, C\_2, C\_2 x C\_2}}
Suppose $G$ is one of the groups in~\eqref{eq:torgroup} and $G\neq C_3,\ C_2,\ C_2\times C_2$.
By~\cite[Proposition~3.2 and Proposition~3.3]{HarronSnowden}, there exists $f,g\in\Q[t]$ coprime with degrees given in Table~\ref{ta:alpT}, such that whenever $E(\Q)_{\mathrm{tors}}\cong G$, then $E$ must be $\Q$-isomorphic to
\[\mathcal{E}_t:y^2=x^3+f(t) x+g(t)\]
for some $t\in\Q$.
The explicit polynomials $f(t)$ and $g(t)$ given in Table~\ref{ta:fgT} can be recovered by dehomogenising the polynomials from~\cite[Proposition~4.3]{Barrios} and~\cite[Tables 4 to 6]{Barrios}.
In Table~\ref{ta:gamT} we give the factorisations of $d(t)=4f(t)^3+27g(t)^2$.
Then we can apply Theorem~\ref{theorem:szpiro1} with $\nu=1$.
Since every irreducible factor of $D_i$ has degree at most $3$ in all cases, we can take $w_i=1$ for all $i$ and $\beta_G =\lambda_{1}(f,g)=\kappa_{1}(f,g)$. The values of \[\beta_G=\frac{12n}{\delta_0+\delta_1+\dots +\delta_{\psi}}\]
are computed in Table~\ref{ta:alpT} from the polynomials in Table~\ref{ta:fgT} and Table~\ref{ta:gamT}, noting that $e_0=\frac{12n}{ m}-\deg d\geq 0$.
This proves Theorem~\ref{theorem:SzpiroTorsheight} when $G\neq C_3,\ C_2,\ C_2\times C_2$.

\subsection{When \texorpdfstring{$G= C_3$}{G=C\_3}}
By~\cite[Lemma~3.5 and Proposition~3.6]{HarronSnowden}, other than exceptional curves of the form $y^2=x^3+b^2$ with $b\in\Z$, any elliptic curve admitting a $3$-torsion point admits an equation of the form
\[y^2=x^3+f(t) x+g(t)\]
with $f(t)=2t-\frac{1}{3}$ and $g(t)=t^2-\frac{2}{3}t+\frac{2}{27}$.
There are $O(H^{\frac{1}{4}})$-many exceptional curves up to height $H$.
Scaling $t$ appropriately transforms $f$ and $g$ to the polynomials given in Table~\ref{ta:fgT}. Then as in the previous case, the case when $G= C_3$ follows from Theorem~\ref{theorem:szpiro1} with $\nu=1$.

\subsection{When \texorpdfstring{$G= C_2\times C_2$}{G=C\_2 x C\_2}}
By~\cite[Proposition~4.2]{HarronSnowden}, an elliptic curve given by
$y^2=x^3+Ax+B$ has full rational $2$-torsion if and only if there exist $u,t\in\Q$, such that $A=u^2f(t)$ and $B=u^3g(t)$, with $f(t)=-\frac{1}{3}(t^2-t+1)$ and $g(t)=\frac{1}{27}(-2t^3+3t^2+3t-2)$.
We compute
\[d(t)=4f(t)^3+27g(t)^2=-t^2 (t - 1)^2.\]
Apply Theorem~\ref{theorem:szpiro1} with $\nu=2$ proves Theorem~\ref{theorem:SzpiroTorsheight} when $G= C_2\times C_2$.

\subsection{When \texorpdfstring{$G= C_2$}{G=C\_2}}
By~\cite[Lemma~5.1]{HarronSnowden}, an elliptic curve given by
$y^2=x^3+Ax+B$ with $A,B\in\Z$ has a point of order $2$ if and only if there exist $a,b\in\Z$, such that $A=a$ and $B=b^3+ab$.
By~\cite[Lemma~5.2]{HarronSnowden}, $A \ll H^{\frac{1}{3}}$ and $B\ll H^{\frac{1}{2}}$, implies that $|a|\ll H^{\frac{1}{3}}$ and $|b|\ll H^{\frac{1}{6}}$.
The discriminant can be expressed as
\[\Disc(E_{A,B})=-16(4a^3+27(b^3+ab)^2)=-16 (a + 3 b^2)^2 (4 a + 3 b^2)\ll H.\]
Theorem~\ref{theorem:main2} allows us to assume that
\[\left|(a + 3 b^2) (4 a + 3 b^2)\right|< \rad(\Disc(E_{A,B}))^{1+\frac{1}{4}\epsilon}\ll N(E_{A,B})^{1+\frac{1}{4}\epsilon}.\]
We can assume by Lemma~\ref{lemma:rebinaryformarea} that
\[\left|(a + 3 b^2) (4 a + 3 b^2)\right|>H^{\frac{2}{3}-\frac{1}{4}\epsilon}.\]
Therefore
\[\Disc(E_{A,B})^{\frac{2}{3}-\frac{1}{4}\epsilon}\ll N(E_{A,B})^{1+\frac{1}{4}\epsilon}.\]
For large enough $H$, we have
\[|\Disc_{\min}(E_{A,B})|\leq |\Disc(E_{A,B})|\leq N(E_{A,B})^{\frac{3}{2}+\epsilon}.\] 

We can assume by Lemma~\ref{lemma:hompolygcd} that $\gcd(\rad(A),\rad(\Disc(E_{A,B})))<H^{\frac{1}{4}\epsilon}$, so 
\[H^{-\frac{1}{4}\epsilon}N(E_{A,B})\ll|(a + 3 b^2) (4 a + 3 b^2)|\ll H^{\frac{2}{3}}.\]
By Lemma~\ref{lemma:rebinaryformarea}, we can assume also that
\[|\Disc_{\min}(E_{A,B})|\geq H^{1-\frac{1}{4}\epsilon}.\]
Combining we have
\[|\Disc_{\min}(E_{A,B})|\geq N(E_{A,B})^{\frac{3}{2}-\epsilon}\]
for large enough $H$.

Collecting all the error terms we see that
\[\#\left\{E/\Q: E(\Q)_{\mathrm{tors}}\cong C_2,\ H(E)\leq H,\ \sigma(E)\notin \left(\frac{3}{2}-\epsilon,\frac{3}{2}+\epsilon\right)\right\}\ll H^{\frac{1}{2}}\frac{(\log\log H)^2}{(\log H)^{\frac{1}{2}}}.\]
Finally
\cite[Theorem~5.6]{HarronSnowden} gives a sufficient lower bound for the denominator
\[\#\left\{E/\Q: E(\Q)_{\mathrm{tors}}\cong C_2,\ H(E)\leq H\right\}\gg H^{\frac{1}{2}}.\]
This completes the proof of Theorem~\ref{theorem:SzpiroTorsheight} when $G= C_2$.

\begingroup\renewcommand\arraystretch{1.5} \tiny
\begin{longtable}{C{1cm}C{6.9cm}C{6.9cm}}
 \caption{The polynomials $f(t)$ and $g(t)$}
\label{ta:fgT}	\\
 \hline
 $G$ &$-3f(t)$ & $\frac{27}{2}g(t)$\\
 \hline \hline
 \endfirsthead
 \caption[]{\emph{continued}}\\
 \hline
 $G$ &$-3f(t)$ & $\frac{27}{2}g(t)$\\
 \hline \hline
 \endhead
 \hline
 \multicolumn{3}{r}{\emph{continued on next page}}
 \endfoot
 \hline
 \endlastfoot
$C_{3}$ & $ -24t+1 $ & $ 216t^2-36t+1$\\\hline
$C_{4}$ & $ 16t^{2}-16t+1 $ & $ ( 8t-1)(8t^{2}+16t-1)$\\\hline
$C_{5}$ & $ t^{4}+12t^{3}+14t^{2}-12t
+1 $ & $(t^{2}+1)(t^{4}
+18t^{3}+74t^{2}-18t+1)$ \\\hline
$C_{6}$ & $( t+3) ( t^{3}+9t^{2}+3t
+3) $& $(t^{2}+6t-3) (t^{4}
+12t^{3}+30t^{2}+36t+9)$ \\\hline
$C_{7}$ & $( t^{2}-t+1) ( t^{6}
+5t^{5}-10t^{4}-15t^{3}+30t^{2}-11t+1) $ & $(t^{12}+6t^{11}-15t^{10}-46t^{9}+174t^{8}-222t^{7}
+273t^{6}-486t^{5}+570t^{4}-354t^{3}+
117t^{2}-18t+1)$ \\\hline
$C_{8}$ & $ t^{8}-16t^{7}+96t^{6}-288t^{5}+480t^{4}-448t^{3}+224t^{2}-64t+16 $ & $
(t^{4}-8t^{3}+16t^{2}-16t+8)(t^{8}
-16t^{7}+96t^{6}-288t^{5}+456t^{4}-352t^{3}+
80t^{2}+32t-8)$\\\hline
$C_{9}$ & $( t^{3}-3t+1) ( t^{9}
-9t^{7}+27t^{6}-45t^{5}+54t^{4}-48t^{3}+27t^{2}-9t+1) $ & $
(t^{18}-18t^{16}+42t^{15}+27t^{14}-306t^{13}
+735t^{12}-1080t^{11}+1359t^{10}-
2032t^{9}+ 3240t^{8}- 4230t^{7}+4128t^{6}
- 2970t^5 + 1557t^4-570t^{3}+
135t^{2}-18t+1)
$ \\\hline
$C_{10}$ & $t^{12}-8t^{11}+16t^{10}+40t^{9}-240t^{8}+432t^{7}-256t^{6}-288t^{5}+720t^{4}-720t^{3}+416t^{2}-
128t+16$ & $(t^{2}-2t+2)(t^{4}-2t^{3}+2)
(t^{4}-2t^{3}-6t^{2}+12t-4)(t^{8}-6t^{7}
+4t^{6}+48t^{5}- 146t^{4}+176t^{3}-104t^{2}
+32t-4)$ 
\\\hline
$C_{12}$ & $( t^{4}-6t^{3}+12t^{2}-12t+6) (t^{12}
-18t^{11}+144t^{10}-684t^{9}+2154t^{8}-4728t^{7}+7368t^{6}-
8112t^{5}+6132t^{4}-3000t^{3}+864t^{2}-144t+24)$ & $(t^{8}-12t^{7}+60t^{6}-168t^{5}+288t^{4}-312t^{3}+216t^{2}-96t+24)(t^{16}-
24t^{15}+264t^{14}+ 8208t^{12}-27696t^{11}+70632t^{10}-138720t^{9}+211296t^{8}-
248688t^{7}+222552t^{6}- 146304t^{5}+65880t^{4}
-17136t^{3}+1008t^{2}+
576t-72)$ \\\hline
$C_{2}\times C_{4}$ & $t^{4}+16t^{3}+80t^{2}+128t+256$ & $(t^{2}+8t-16)(
t^{2}+8t+8) (t^{2}+8t+32)$\\\hline
$C_{2}\times C_{6}$ & $( 21t^{2}-6t+1) ( 6861t^{6}
-2178t^{5}-825t^{4}+180t^{3}+75t^{2}-18t+1) $& $(183t^{4}-36t^{3}-30t^{2}+12t-1)(
393t^{4}-156t^{3}+30t^{2}-12t+1)(759t^{4}-
228t^{3}-30t^{2}+ 12t-1)$
\\\hline
$C_{2}\times C_{8}$ & $t^{16}+32t^{15}+448t^{14}+3584t^{13}+17664t^{12}
+51200t^{11}+51200t^{10}-237568t^{9}- 1183744t^{8}-1900544t^{7}+3276800t^{6} +26214400t^{5}+
72351744t^{4}+117440512t^{3}+ 117440512t^{2}+67108864t+16777216$
& $(t^{8}+16t^{7}+96t^{6}+256t^{5}-256t^{4}
-4096t^{3}-12288t^{2}-16384t-8192)
( t^{8}+16t^{7}+ 96t^{6}+256t^{5}+128t^{4}-1024t^{3}-3072t^{2}-4096t-2048)(t^{8}+
16t^{7}+96t^{6}+ 256t^{5}+ 512t^{4}+2048t^{3}
+6144t^{2}+8192t+4096)$
\end{longtable}

\begin{longtable}{C{1cm}C{8cm}}
 \caption{The polynomials $d(t)$}
\label{ta:gamT}	
\\
 \hline
 $G$ & $-\frac{1}{2^8}d(t) $ \\
 \hline \hline
 \endfirsthead
 \caption[]{\emph{continued}}\\
 \hline
 $G$ & $-\frac{1}{2^8}d(t) $ \\
 \hline \hline
 \endhead
 \hline
 \multicolumn{2}{r}{\emph{continued on next page}}
 \endfoot
 \hline
 \endlastfoot
 $C_{3}$ & $t^{3}(-27t+1) $ \\\hline
 $C_{4}$ & $t^{4}(-16t+1) $ \\\hline
$C_{5}$ & $ t ^{5}( -t^{2}-11t+1)
$ \\\hline
$C_{6}$ & $t^{2}( t+9) ( t+1) ^{3}
$ \\\hline
$C_{7}$ & $ t ^{7}( -t+1) ^{7}(
t^{3}+5t^{2}-8t+1) $ \\\hline
$C_{8}$ & $t^{2}( t-2) ^{4}( t-1)
^{8}( t^{2}-8t+8) $ \\\hline
$C_{9}$ & $ t^{9}( -t+1) ^{9}(
t^{2}-t+1) ^{3}( t^{3}+3t^{2}-6t+1) $ \\\hline
$C_{10}$ & $t^{5}( t-2) ^{5}(
t-1) ^{10}( t^{2}+2t-4) ( t^{2}-3t+1) ^{2}$ \\\hline
$C_{12}$ & $t^{2}( t-2) ^{6}( t-1)
^{12}( t^{2}-6t+6) ( t^{2}-2t+2)
^{3}( t^{2}-3t+3) ^{4}$ \\\hline
$C_{2}\times C_{4}$ & $t^{2}( t+8)^{2}( t+4)^{4}$ \\\hline
$C_{2}\times C_{6}$ & $( 2t) ^{6}( -9t+1) ^{2}(
-3t+1) ^{2}( 3t+1) ^{2}( -5t+1) ^{6}(
-t+1) ^{6}$ \\\hline
$C_{2}\times C_{8}$ & $( 2t) ^{8}( t+2) ^{8}(
t+4) ^{8}( t^{2}-8)^{2}( t^{2}+8t+8
)^{2}( t^{2}+4t+8)^{4}$ 
\end{longtable}

\begin{longtable}{C{1cm}C{1cm}C{1cm}C{1cm}C{1cm}C{1cm}C{1cm}C{1cm}C{1cm}}
 \caption{Data for computing $\beta_G$}
\label{ta:alpT}	
\\
 \hline
 $G$ &$\deg f$ & $\deg g$ & $n$ & $m$ & $\frac{12n}{m+1}$ & $\mathbf{1}_{e_0\geq 1}$ & $\frac{1}{m}\sum_{i=1}^{\psi} \delta_i$ & $\beta_G$\\
 \hline \hline
 \endfirsthead
 \caption[]{\emph{continued}}\\
 \hline
 $G$ &$\deg f$ & $\deg g$ & $n$ & $m$ & $\frac{12n}{m+1}$ & $\mathbf{1}_{e_0\geq 1}$ & $\frac{1}{m}\sum_{i=1}^{\psi} \delta_i$ & $\beta_G$\\
 \hline \hline
 \endhead
 \hline
 \multicolumn{9}{r}{\emph{continued on next page}}
 \endfoot
 \hline
 \endlastfoot
$C_{3}$ 		& $1$ & $2$ & $1$ & $3$ & $3$ & $0$ & $2$ & $2$\\\hline
$C_{4}$ 		& $2$ & $3$ & $1$ & $2$ & $4$ & $1$ & $2$ & $\frac{12}{5}$\\\hline
$C_{5}$		& $4$ & $6$ & $1$ & $1$ & $6$ & $1$ & $3$ & $3$\\\hline
$C_{6}$ 		& $4$ & $6$ & $1$ & $1$ & $6$ & $1$ & $3$ & $3$\\\hline
$C_{7}$ 		& $8$ & $12$ & $2$ & $1$ & $12$ & $1$ & $5$ & $4$\\\hline
$C_{8}$ 		& $8$ & $12$ & $2$ & $1$ & $12$ & $1$ & $5$ & $4$\\\hline
$C_{9}$ 		& $12$ & $18$ & $3$ & $1$ & $18$ & $1$& $7$ & $\frac{9}{2}$\\\hline
$C_{10}$ 		& $12$ & $18$ & $3$ & $1$ & $18$ & $1$& $7$ & $\frac{9}{2}$\\\hline
$C_{12}$ 		& $16$ & $24$ & $4$ & $1$ & $24$ & $1$& $9$ & $\frac{24}{5}$\\\hline
$C_{2}\times C_{4}$ & $4$ & $6$ & $1$ & $1$ & $6$ & $1$& $3$ & $3$\\\hline
$C_{2}\times C_{6}$ & $8$ & $12$ & $2$ & $1$ & $12$ & $0$& $6$ & $4$\\\hline
$C_{2}\times C_{8}$ & $16$ & $24$ & $4$ & $1$ & $24$ & $1$& $9$ & $\frac{24}{5}$
\end{longtable}
\endgroup

\bibliographystyle{abbrv}

\end{document}